\documentclass[12pt]{article}
\usepackage[utf8]{inputenc}
\usepackage{amsthm}
\usepackage{amsmath}
\usepackage{mathtools}
\usepackage{bbm}
\usepackage{amsfonts}
\usepackage{amssymb}
\usepackage[left=2.5cm,right=2.5cm,top=2cm,bottom=2cm]{geometry}

\usepackage[dvipsnames]{xcolor}

\usepackage{tabto}
\TabPositions{2 cm, 4 cm, 6 cm, 8 cm}

\usepackage{tikz-cd}

\usepackage{tikz}
\usetikzlibrary{arrows, automata, positioning}

\usepackage{tocloft}
\usepackage{appendix}

\usepackage{enumerate}

\usepackage{hyperref}
 \newcommand{\Manoa}{M\=anoa }
\newcommand{\Hawaii}{Hawai\kern.05em`\kern.05em\relax i }
\newtheorem{theorem}{Theorem}[section]
\newtheorem{corollary}[theorem]{Corollary}
\newtheorem{lemma}[theorem]{Lemma}

\newtheorem{proposition}[theorem]{Proposition}

\theoremstyle{definition}
\newtheorem{definition}[theorem]{Definition}

\newtheorem*{definition*}{Definition}

\theoremstyle{remark}

\newtheorem*{lemma*}{Lemma}
\newtheorem*{proposition*}{Proposition}
\newtheorem*{theorem*}{Theorem}
\newtheorem*{corollary*}{Corollary}
\newtheorem*{claim*}{Claim}

\theoremstyle{definition}

\newtheorem{question}[theorem]{Question}

\newcommand{\im}{\operatorname{im}}

\newcommand{\HH}{\operatorname{H}}

\newcommand{\CC}{\operatorname{C}}
\newcommand{\ZZ}{\operatorname{Z}}
\newcommand{\BB}{\operatorname{B}}

\newcommand{\F}{\operatorname{F}}

\begin{document}

\title{Braided Thompson groups with and without quasimorphisms}
\author{Francesco Fournier-Facio, Yash Lodha and Matthew C. B. Zaremsky}
\date{\today}
\maketitle

\begin{abstract}
We study quasimorphisms and bounded cohomology of a variety of braided versions of Thompson groups. Our first main result is that the Brin--Dehornoy braided Thompson group $bV$ has an infinite-dimensional space of quasimorphisms and thus infinite-dimensional second bounded cohomology. This implies that despite being perfect, $bV$ is not uniformly perfect, in contrast to Thompson's group $V$. We also prove that relatives of $bV$ like the ribbon braided Thompson group $rV$ and the pure braided Thompson group $bF$ similarly have an infinite-dimensional space of quasimorphisms. Our second main result is that, in stark contrast, the close relative of $bV$ denoted $\widehat{bV}$, which was introduced concurrently by Brin, has trivial second bounded cohomology. This makes $\widehat{bV}$ the first example of a left-orderable group of type $\F_\infty$ that is not locally indicable and has trivial second bounded cohomology. This also makes $\widehat{bV}$ an interesting example of a subgroup of the mapping class group of the plane minus a Cantor set that is non-amenable but has trivial second bounded cohomology, behaviour that cannot happen for finite-type mapping class groups.
\end{abstract}

\section{Introduction}

The braided Thompson group $bV$ was introduced independently by Brin \cite{brin07} and Dehornoy \cite{dehornoy}, as a braided version of the classical Thompson group $V$. This group and its relatives have proven to be important objects in geometric group theory, in particular thanks to their connections to big mapping class groups. Recall that a surface is said to be \emph{of infinite type} if its fundamental group is not finitely generated, and to such a surface one can associate a mapping class group, in the same way as for finite-type surfaces; such mapping class groups are called \emph{big}. As an example of the connection, certain braided Thompson groups are dense in the big mapping class group of a compact surface minus a Cantor set \cite[Corollary~3.20]{skipperwustab}, and hence serve as finitely generated ``approximations'' of these big mapping class groups. For more on connections between braided Thompson groups and big mapping class groups see, e.g., \cite{abfpw,aramayona21,funar04,funar08,funar11,genevois21}.

In this paper we are concerned with the question of which braided Thompson groups have infinite-dimensional space of quasimorphisms, or second bounded cohomology, and which do not. A function $q : \Gamma \to \mathbf{R}$ is called a \emph{quasimorphism} if the quantity $|q(g) + q(h) - q(gh)|$ is uniformly bounded: its supremum is called the \emph{defect} of $q$ and is denoted by $D(q)$. We denote by $Q(\Gamma)$ the space of quasimorphisms of $\Gamma$, modulo bounded functions
(sometimes this notation is used to denote the space of homogeneous quasimorphisms, which is canonically isomorphic \cite[2.2.2]{calegari}).
We may sometimes colloquially refer to a group as having ``no quasimorphisms'' if it only has bounded ones. The objects $Q(\Gamma)$ are of great interest in dynamics, geometric group theory, geometric topology and symplectic geometry. For example, they are intimately connected with bounded cohomology \cite{frigerio} and stable commutator length \cite{calegari}. In this context, Thompson-like groups have played an important role: for instance they have repeatedly served as the first finitely presented examples achieving certain values of stable commutator length \cite{ghysserg, zhuang, FFL1}.

In addition to $bV$, we also inspect the ribbon braided Thompson group $rV$, the pure braided Thompson group $bF$, the kernel $bP$ of the projection $bV\to V$, and most importantly the group $\widehat{bV}$, which was introduced by Brin along with $bV$ in \cite{brin07}. One can view $\widehat{bV}$ as a braided analogue of a Cantor set point stabiliser in $V$. See Section~\ref{sec:braided_thompson} for the definitions of all these braided Thompson groups. The group $\widehat{bV}$, despite its strong similarities to $bV$, has extremely different behaviour when it comes to quasimorphisms and bounded cohomology, as our two main results make clear. \\

Our main results are as follows.

\begin{theorem}\label{thm:main1}
For $\Gamma$ any of the braided Thompson groups $bV$, $rV$, $bF$ or $bP$, the space $Q(\Gamma)$ is infinite-dimensional, and thus also the second bounded cohomology $\HH_b^2(\Gamma)$ is infinite-dimensional.
\end{theorem}

\begin{theorem}\label{thm:main2}
We have $\HH^2_b(\widehat{bV})=0$.
\end{theorem}

Here $\HH^2_b(\Gamma)$ denotes the second bounded cohomology of a group $\Gamma$, with trivial real coefficients. This invariant was introduced by Johnson and Trauber in the context of Banach algebras \cite{johnson} and has since become a fundamental tool in geometric topology \cite{gromov}, dynamics \cite{ghys} and rigidity theory \cite{burgermonod}. For every group $\Gamma$ there is a map $Q(\Gamma) \to \HH^2_b(\Gamma)$, whose kernel is the space of real-valued homomorphisms (Proposition \ref{prop:hbqm}). Using this, Theorem \ref{thm:main2}, together with the fact that the abelianisation of $\widehat{bV}$ is isomorphic to $\mathbf{Z}$ (Corollary~\ref{cor:ablnz_to_1}), implies:

\begin{corollary}\label{cor:main}
$Q(\widehat{bV})$ is one-dimensional, spanned by the abelianisation of $\widehat{bV}$.
\end{corollary}

One consequence of Theorem~\ref{thm:main1} is that, despite being perfect \cite{zaremsky18}, $bV$ is not uniformly perfect (Corollary \ref{cor:up}). Recall that a group $\Gamma$ is \emph{uniformly perfect} if there exists $N\in\mathbf{N}$ such that every element in $\Gamma$ can be written as a product of at most $N$ commutators. This is in contrast to the fact that Thompson's group $V$ is uniformly perfect, and even uniformly simple \cite{uniformlysimple} (in fact, $\HH^2_b(V) = 0$ \cite[6.3]{binate}). Since $bV$ is not uniformly perfect, the following natural question emerges:

\begin{question}\label{quest:scl}
Which elements of $bV$ have non-zero stable commutator length?
\end{question}

A characterisation of this phenomenon in (finite-type) mapping class groups was given in \cite{mcgscl}, and see \cite{field} for some related results for big mapping class groups. \\

Theorem~\ref{thm:main2} has interesting consequences for subgroups of big mapping class groups. Pioneering work of Bestvina and Fujiwara \cite{mcgqm} showed that every subgroup of a (finite-type) mapping class group is either virtually abelian or has infinite-dimensional $Q(\Gamma)$ (see also \cite{mcgscl}). This can be viewed as a sort of Tits-like alternative, since every quasimorphism on an amenable group is at a bounded distance from a homomorphism \cite{brooks}, whereas groups with hyperbolic features typically have an infinite-dimensional space of quasimorphisms \cite{brooks, epsteinfujiwara, hullosin}. The question of whether something similar happens for the big mapping class group $MCG(\mathbf{R}^2\setminus K)$ for $K$ a Cantor set was listed in the AIM problem list on big mapping class groups \cite[Question 4.7]{questions}. Namely, it is asked whether every subgroup $\Gamma \leq MCG(\mathbf{R}^2\setminus K)$ is either amenable or has infinite-dimensional $Q(\Gamma)$. Theorem~\ref{thm:main2} provides a negative answer to this, since $\widehat{bV}$ is non-amenable (by virtue of containing braid groups), and embeds in $MCG(\mathbf{R}^2\setminus K)$ (see Section \ref{sec:quasis}).

In fact we should point out that a negative answer to this question was already ``almost'' available in the literature. Indeed, by a result of Calegari and Chen \cite{calegarichen} every countable circularly orderable group $\Gamma$ embeds in $MCG(\mathbf{R}^2 \setminus K)$, and there are plenty of countable circularly orderable groups that are non-amenable and have a finite-dimensional space of quasimorphisms, or no quasimorphisms at all \cite{calegariPL, zhuang, FFL1}. The most straightforward example is probably Thompson's group $T$, which has no quasimorphisms by virtue of being uniformly perfect (and even uniformly simple, see e.g., \cite{Tus}). In fact, when the groups are even left-orderable, many of them have vanishing second bounded cohomology \cite{FFL1, FFL2}, and sometimes even vanishing bounded cohomology in every positive degree \cite{monodF}. As a remark, since the examples coming from the procedure in \cite{calegarichen} act on the plane by fixing a radial coordinate and acting by rotations, which is really a ``one-dimensional'' picture, one can view $\widehat{bV}$ as providing the first truly ``two-dimensional'' example, i.e., one involving genuine braids. \\

In order to prove Theorem~\ref{thm:main1}, we generally follow the approach used by Bavard in \cite{bavard} (see \cite{bavard_english} for an English translation) to show that $MCG(\mathbf{R}^2\setminus K)$ has an infinite-dimensional space of quasimorphisms. Her proof in turn makes use of the approach of Bestvina and Fujiwara in \cite{mcgqm} to finite-type mapping class groups, following suggestions of Calegari from a blogpost \cite{blogpost}. Bavard's result prompted the study of analogues of curve graphs for big mapping class groups, and arguably initiated the recent surge of interest in big mapping class groups; see \cite{survey} for more on the history of big mapping class groups.

In the course of proving Theorem~\ref{thm:main2}, we also prove that $\widehat{bV}$ is of type $\F_\infty$, meaning it has a classifying space with finitely many cells in each dimension (Corollary \ref{cor:F_infty}); this is a stronger property than finite generation and finite presentability. It is known that $bV$ and thus $\widehat{bV}$ are left-orderable \cite{ishida18}, and that $\widehat{bV}$ contains a copy of $bV$ (see Definition~\ref{def:right_depth}), which is finitely generated and perfect \cite{zaremsky18}. Therefore $\widehat{bV}$ serves as the first example of a group with the following properties:

\begin{corollary}\label{cor:why_bVhat_is_cool}
The group $\widehat{bV}$ is a left-orderable group of type $\F_\infty$ that is not locally indicable and has vanishing second bounded cohomology.\qed
\end{corollary}

A finitely generated group is \emph{indicable} if it admits a homomorphism onto $\mathbf{Z}$. A group is \emph{locally indicable} if each of its finitely generated subgroups is indicable. The combination of these properties is interesting because it shows that in the celebrated Witte Morris Theorem \cite{wm} the hypothesis of amenability cannot be weakened to the vanishing of second bounded cohomology. The first finitely generated examples were found in \cite{FFL1}: this was the first step towards finding examples with the additional property of being non-indicable \cite{FFL2}, answering a question of Navas \cite{questions_navas}. Since $\widehat{bV}$ is indicable, the existence of type $\F_\infty$ examples with these stronger properties is still open.

We will always stick to the ``$n=2$ case'' to avoid getting bogged down in notation, but the reader should note that all of our results can be adapted to the braided Higman--Thompson groups $bV_n$ (as in \cite{aroca21,skipperwu}) and their analogous subgroups $\widehat{bV_n}$, with appropriate small modifications to the arguments. It would be interesting to try and adapt our arguments to other, more complicated Thompson-like groups related to asymptotically rigid mapping class groups, e.g., for positive genus surfaces \cite{aramayona21}, or for higher-dimensional manifolds \cite{abfpw}. \\

\textbf{Acknowledgments.} We wish to thank Javier Aramayona, Mladen Bestvina, Peter Feller, Marissa Loving, Nick Vlamis for useful discussions, and the anonymous referee for helpful suggestions. The first author was supported by an ETH Z\"urich Doc.Mobility Fellowship. The second author was supported by a START-projekt grant Y-1411 of the Austrian Science Fund. The third author was supported by grant \#635763 from the Simons Foundation.

\section{Braided Thompson groups}\label{sec:braided_thompson}

The first braided Thompson group, which we denote by $bV$ and which has also been denoted by $BV$, $V_{\operatorname{br}}$, and $\operatorname{br}\!V$ in the literature, was introduced independently by Brin \cite{brin07} and Dehornoy \cite{dehornoy}, as a braided version of Thompson's group $V$. Other braided Thompson groups include the ``$F$-like'' pure braided Thompson groups $bF$ \cite{brady08}, various ``$T$-like'' braided Thompson groups \cite{funar08,funar11,witzel19}, braided Higman--Thompson groups $bV_n$ \cite{aroca21,skipperwu}, braided Brin--Thompson groups $sV_{\operatorname{br}}$ \cite{spahn}, the ``ribbon braided'' Thompson group $rV$ \cite{thumann} and braided R\"over--Nekrashevych groups $\operatorname{br}\!V_d(G)$ \cite{skipper}. Most relevant to our purposes here is a close relative $\widehat{bV}$ of $bV$, which was also introduced by Brin in \cite{brin07} (there denoted $\widehat{BV}$), and realised up to isomorphism as a concrete subgroup of $bV$ by Brady--Burillo--Cleary--Stein \cite{brady08}; see also \cite{burillo_cleary}.

Let us recall the definitions of $bV$ and $\widehat{bV}$, using the standard braided tree pair model, as in \cite{brady08,zaremsky18}. By a \emph{tree} we will always mean a finite, rooted, planar, binary tree. An element of $bV$ is represented by a \emph{representative triple} $(T_-,\beta,T_+)$, where $T_-$ is a tree, $T_+$ is a tree with the same number of leaves as $T_-$, say $n$, and $\beta$ is a braid in $B_n$. Elements of $bV$ are equivalence classes $[T_-,\beta,T_+]$ of representative triples, where the equivalence relation is given by the notion of expansion, which we now describe. 

First, denote by $\rho_n\colon B_n\to S_n$ the usual map from the braid group to the symmetric group, recording how the numbering of the strands at the bottom changes when the strands move to the top. (We may write $\rho$ for $\rho_n$ when we do not need to care about $n$.) Let $(T_-,\beta,T_+)$ be a representative triple, say with $T_\pm$ having $n$ leaves and $\beta\in B_n$, and let $1\le k\le n$. Let $T_+'$ be the tree obtained from $T_+$ by adding a caret to the $k$th leaf, let $T_-'$ be the tree obtained from $T_-$ by adding a caret to the $\rho_n(\beta)(k)$th leaf, and let $\beta'\in B_{n+1}$ be the braid obtained from $\beta$ by bifurcating the $k$th strand (counting at the bottom) into two parallel strands.

\begin{definition}[Expansion, equivalence]
With the above setup, call $(T_-',\beta',T_+')$ the $k$th \emph{expansion} of $(T_-,\beta,T_+)$. Declare that two representative triples are equivalent if one is an expansion of the other, and extend this to generate an equivalence relation on the set of representative triples.
\end{definition}

The elements of the group $bV$ are the equivalence classes $[T_-,\beta,T_+]$, and the group operation is described as follows. Given two elements $[T_-,\beta,T_+]$ and $[U_-,\gamma,U_+]$, up to expansions we can assume that $T_+=U_-$. Now we define
\[
[T_-,\beta,T_+][T_+,\gamma,U_+] \coloneqq [T_-,\beta\gamma,U_+]\text{.}
\]

Some immediate subgroups of $bV$ include Thompson's group $F$, which is the subgroup of elements of the form $[T_-,1,T_+]$, and the pure braided Thompson group $bF$, which is the subgroup of elements of the form $[T_-,\beta,T_+]$ for $\beta$ a pure braid. We will also be especially interested in the following subgroup:

\begin{definition}[The group $\widehat{bV}$]\label{def:bVhat}
For each $n\in\mathbf{N}$, let $\widehat{B}_n$ denote the standard copy of $B_{n-1}$ inside $B_n$ which only braids the first $(n - 1)$ strands. Note that if $(T_-',\beta',T_+')$ is an expansion of $(T_-,\beta,T_+)$, say with $\beta\in B_n$ and $\beta'\in B_{n+1}$, then we have $\beta\in \widehat{B}_n$ if and only if $\beta'\in\widehat{B}_{n+1}$. Thus, the equivalence classes $[T_-,\beta,T_+]$ for $\beta\in\widehat{B}_n$ form a well defined subgroup of $bV$, denoted $\widehat{bV}$.
\end{definition}

There is a convenient way to picture elements of $bV$ as (equivalence classes of) so called \emph{strand diagrams}. For an element $[T_-,\beta,T_+]$, we picture $T_+$ upside-down and below $T_-$, with $\beta$ connecting the leaves of $T_+$ up to the leaves of $T_-$. See Figure~\ref{fig:expansion} for an example of an element of $bV$, and an expansion.

\begin{figure}[htb]
\centering
\begin{tikzpicture}[line width=0.8pt,scale=0.8]
 \draw (-0.5,-1.5) to [out=-90, in=90] (-1.5,-3.5);
 \draw[line width=4pt, white] (-1.5,-1.5) to [out=-90, in=90] (0.5,-3.5);
 \draw (-1.5,-1.5) to [out=-90, in=90] (0.5,-3.5);
 \draw[line width=4pt, white] (0.5,-1.5) to [out=-90, in=90] (-0.5,-3.5);
 \draw (0.5,-1.5) to [out=-90, in=90] (-0.5,-3.5);
 \draw (1.5,-1.5) -- (1.5,-3.5);
  
 \filldraw (0,0) circle (1.5pt)   (-1.5,-1.5) circle (1.5pt)   (1.5,-1.5) circle (1.5pt)   (-1,-1) circle (1.5pt)   (1,-1) circle (1.5pt)   (-0.5,-1.5) circle (1.5pt)   (0.5,-1.5) circle (1.5pt);
 \draw (-1.5,-1.5) -- (0,0) -- (1.5,-1.5)   (-0.5,-1.5) -- (-1,-1)   (0.5,-1.5) -- (1,-1);
 
 \filldraw (0,-5) circle (1.5pt)   (-1.5,-3.5) circle (1.5pt)   (1.5,-3.5) circle (1.5pt)   (-1,-4) circle (1.5pt)   (-0.5,-4.5) circle (1.5pt)   (-0.5,-3.5) circle (1.5pt)   (0.5,-3.5) circle (1.5pt);;
 \draw (-1.5,-3.5) -- (0,-5) -- (1.5,-3.5)   (-0.5,-3.5) -- (-1,-4)   (0.5,-3.5) -- (-0.5,-4.5);
 
 \node at (2.5,-2.5) {$=$};
 
 \begin{scope}[xshift=5cm]
 \draw (-0.5,-1.5) to [out=-90, in=90] (-1.5,-3.5);
 \draw[line width=3pt, white] (-1.5,-1.5) to [out=-90, in=90] (0.5,-3.5);
 \draw (-1.5,-1.5) to [out=-90, in=90] (0.5,-3.5);
 \draw[line width=4pt, white] (0.25,-1.75) to [out=-90, in=90] (-0.75,-3.25);
 \draw (0.25,-1.75) to [out=-90, in=90] (-0.75,-3.25);
 \draw[line width=4pt, white] (0.75,-1.75) to [out=-90, in=90] (-0.25,-3.25);
 \draw (0.75,-1.75) to [out=-90, in=90] (-0.25,-3.25);
 \draw (1.5,-1.5) -- (1.5,-3.5);
  
 \filldraw (0,0) circle (1.5pt)   (-1.5,-1.5) circle (1.5pt)   (1.5,-1.5) circle (1.5pt)   (-1,-1) circle (1.5pt)   (1,-1) circle (1.5pt)   (-0.5,-1.5) circle (1.5pt)   (0.5,-1.5) circle (1.5pt)   (0.25,-1.75) circle (1.5pt)   (0.75,-1.75) circle (1.5pt);
 \draw (-1.5,-1.5) -- (0,0) -- (1.5,-1.5)   (-0.5,-1.5) -- (-1,-1)   (0.5,-1.5) -- (1,-1)   (0.25,-1.75) -- (0.5,-1.5) -- (0.75,-1.75);
 
 \filldraw (0,-5) circle (1.5pt)   (-1.5,-3.5) circle (1.5pt)   (1.5,-3.5) circle (1.5pt)   (-1,-4) circle (1.5pt)   (-0.5,-4.5) circle (1.5pt)   (-0.5,-3.5) circle (1.5pt)   (0.5,-3.5) circle (1.5pt)   (-0.25,-3.25) circle (1.5pt)   (-0.75,-3.25) circle (1.5pt);
 \draw (-1.5,-3.5) -- (0,-5) -- (1.5,-3.5)   (-0.5,-3.5) -- (-1,-4)   (0.5,-3.5) -- (-0.5,-4.5)   (-0.75,-3.25) -- (-0.5,-3.5) -- (-0.25,-3.25);
 \end{scope}
\end{tikzpicture}
\caption{An element $[T_-,\beta,T_+]$ of $bV$. We draw $T_+$ upside down, with $\beta$ as a braid from the leaves of $T_+$ up to the leaves of $T_-$. We have $\rho_4(\beta)=(1~2~3)$, so $\rho_4(\beta)(2)=3$. Thus to perform the $2$nd expansion, we add a caret to the $2$nd leaf of $T_+$, a caret to the $3$rd leaf of $T_-$, and bifurcate the $2$nd strand of $\beta$ (counting from the bottom) into two strands. Note that this element lies in the subgroup $\widehat{bV}$ since the rightmost strand does not braid with any of the others.}
\label{fig:expansion}
\end{figure}
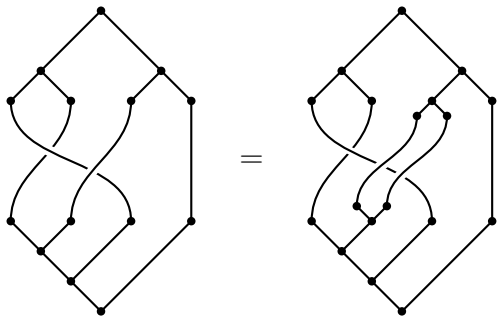

To accurately model the equivalence relation coming from expansion, and the group operation, which amounts to stacking strand diagrams, some equivalences between strand diagrams naturally emerge. The three key equivalences are shown in Figure~\ref{fig:strand_moves}.

\begin{figure}[htb]
\centering
\begin{tikzpicture}[line width=0.8pt,scale=0.8]

\draw (0,0) -- (0,-0.5) -- (-0.5,-1) -- (-0.5,-1.5) -- (0,-2) -- (0,-2.5)   (0,-0.5) -- (0.5,-1) -- (0.5,-1.5) -- (0,-2);
\filldraw (0,0) circle (1.5pt)   (0,-0.5) circle (1.5pt)   (-0.5,-1) circle (1.5pt)   (-0.5,-1.5) circle (1.5pt)   (0,-2) circle (1.5pt)   (0.5,-1) circle (1.5pt)   (0.5,-1.5) circle (1.5pt)   (0,-2.5) circle (1.5pt);

\node at (1,-1.25) {$=$};

\draw (1.5,0) -- (1.5,-2.5);
\filldraw (1.5,0) circle (1.5pt)   (1.5,-2.5) circle (1.5pt);

\begin{scope}[xshift=5.5cm]
\draw (0,0) -- (0.5,-0.5) -- (0.5,-1.5) -- (0,-2)   (0.5,-0.5) -- (1,0)   (0.5,-1.5) -- (1,-2);
\filldraw (0,0) circle (1.5pt)   (1,0) circle (1.5pt)   (0.5,-0.5) circle (1.5pt)   (0,-2) circle (1.5pt)   (0.5,-1.5) circle (1.5pt)   (1,-2) circle (1.5pt);

\node at (1.5,-1) {$=$};

\draw (2,0) -- (2,-2)   (3,0) -- (3,-2);
\filldraw (2,0) circle (1.5pt)   (2,-2) circle (1.5pt)   (3,0) circle (1.5pt)   (3,-2) circle (1.5pt);
\end{scope}

\begin{scope}[xshift=2cm,yshift=-4cm]
\draw (0,0) to [out=-90, in=90] (1,-2);
\draw[line width=4pt, white] (1,0) to [out=-90, in=90] (0,-2);
\draw (1,0) to [out=-90, in=90] (0,-2);
\draw (-0.5,-2.5) -- (0,-2) -- (0.5,-2.5);
\filldraw (0,0) circle (1.5pt)   (1,0) circle (1.5pt)   (0,-2) circle (1.5pt)   (1,-2) circle (1.5pt)   (-0.5,-2.5) circle (1.5pt)   (0.5,-2.5) circle (1.5pt);

\node at (2,-1) {$=$};

 \begin{scope}[xshift=3cm]
 \draw (0,0) to [out=-90, in=90] (2,-2);
 \draw[line width=4pt, white] (1,0) to [out=-90, in=90] (0,-2);
 \draw (1,0) to [out=-90, in=90] (0,-2);
 \draw[line width=4pt, white] (2,0) to [out=-90, in=90] (1,-2);
 \draw (2,0) to [out=-90, in=90] (1,-2);
 
 \draw (1,0) -- (1.5,0.5) -- (2,0);
 \filldraw (0,0) circle (1.5pt)   (1,0) circle (1.5pt)   (2,0) circle (1.5pt)   (1.5,0.5) circle (1.5pt)   (0,-2) circle (1.5pt)   (1,-2) circle (1.5pt)   (2,-2) circle (1.5pt);
 \end{scope}
\end{scope}
  
\end{tikzpicture}
\caption{The three key equivalences for strand diagrams, which can occur anywhere inside a strand diagram representing an element of $bV$. Further equivalences are obtained by combining these, and by rotating and reflecting the bottom one.}
\label{fig:strand_moves}
\end{figure}
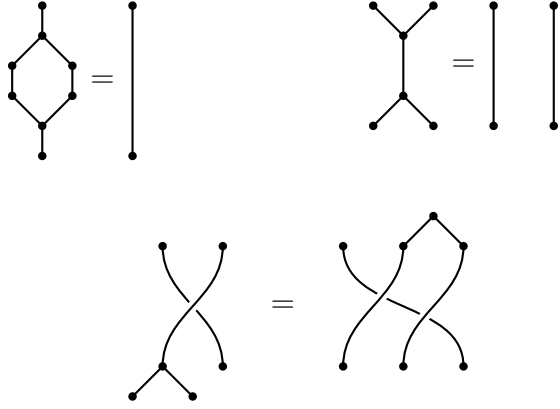

See \cite{brady08,zaremsky18} for more details.

\subsection{Using pure braids and ribbon braids}\label{ss:pure_and_ribbon}

Some more subgroups of $bV$ arise when we restrict to pure braids. As before, consider the standard projection $B_n\to S_n$ from the braid group $B_n$ to the symmetric group $S_n$. The kernel of this map is the \emph{pure braid group} $PB_n$. This leads us to the following definition of the pure braided Thompson group $bF$.

\begin{definition}[The group $bF$]\label{def:bF}
If $(T_-',\beta',T_+')$ is an expansion of $(T_-,\beta,T_+)$ then $\beta'$ is pure if and only if $\beta'$ is pure, so the equivalence classes $[T_-,\beta,T_+]$ for $\beta\in PB_n$ form a well defined subgroup of $bV$, denoted $bF$.
\end{definition}

Note that $bF$ is not normal in $bV$, but the following related subgroup is:

\begin{definition}[The group $bP$]\label{def:bP}
Let $bP$ denote the subgroup of $bV$ consisting of all $[T,\beta,T]$ such that $\beta$ is pure.
\end{definition}

For reference, the group $bP$ was denoted by $PBV$ in \cite{brady08} and by $P_{\operatorname{br}}$ in \cite{zaremsky18}. Note that the trees in $[T,\beta,T]$ must be the same, so $bP$ is strictly smaller than $bF$. The quotient $bV/bP$ is isomorphic to Thompson's group $V$, and the quotient $bF/bP$ is isomorphic to Thompson's group $F$ \cite{brady08}. More precisely, upon passing to a quotient with kernel $bP$, the elements of $bV$ change from being represented by triples $(T_-,\beta,T_+)$ for $\beta\in B_n$ to being represented by triples $(T_-,\sigma,T_+)$ for $\sigma \in S_n$. The notion of expansion has an obvious analogue for permutations, and we get equivalence classes $[T_-,\sigma,T_+]$, which are the elements of $V$. The image of $\widehat{bV}$ under the projection $bV\to V$ is a group called $\widehat{V}$, which was also considered in \cite{brin07}. The kernel $\widehat{bV}\cap bP$ of the projection $\widehat{bV}\to\widehat{V}$ has the following interesting property, which will be useful later:

\begin{lemma}\label{lem:onto_bP}
There is an epimorphism $\widehat{bV}\cap bP \to bP$.
\end{lemma}

\begin{proof}
An element of $\widehat{bV}\cap bP$ is of the form $[T,\beta,T]$ for $\beta$ a pure braid in which the rightmost strand does not braid with any of the others. Let $T'$ be the subtree of $T$ whose root is the left child of the root of $T$ (or if $T$ is trivial just take $T'$ to also be trivial). Let $\beta'$ be the (pure) braid obtained from $\beta$ by deleting any strands corresponding to leaves of $T$ that are closer to the right child of the root than the left (so in particular in the non-trivial case this includes the rightmost strand). Intuitively, $[T',\beta',T']$ is obtained by taking just the ``left part'' of $[T,\beta,T]$, from the point of view of the root of $T$. This operation is well defined up to expansions, and yields a well defined homomorphism $[T,\beta,T]\to [T',\beta',T']$ from $\widehat{bV}\cap bP$ to $bP$, which is clearly surjective.
\end{proof}

Finally, let us discuss a ``twisted'' version of $bV$, called the \emph{ribbon braided Thompson group} $rV$. This arises by treating the strands in a strand diagram as ribbons, which are allowed to twist. This first appeared officially in work of Thumann \cite[Subsection~3.5.3]{thumann}, where he proved $rV$ (there denoted $RV$) is of type $\F_\infty$. The idea of using ribbons to represent strands in $bV$ was actually already present in Brin's original paper \cite{brin07}, but without twisting. We will mostly follow the approach from \cite[Example~4.2]{zaremsky_users_guide}, which uses the notion of cloning systems from \cite{witzel18} to provide a framework for elements of $rV$ similar to the one we are using here for $bV$. An element of $rV$ is represented by a triple $(T_-,\beta(m_1,\dots,m_n),T_+)$ where $T_-$ and $T_+$ are trees with $n$ leaves and $\beta(m_1,\dots,m_n)\in B_n \wr \mathbf{Z}$. More precisely, $\beta\in B_n$, $m_1,\dots,m_n\in\mathbf{Z}$, and $B_n\wr \mathbf{Z}$ denotes the wreath product $B_n \ltimes \mathbf{Z}^n$ with the action induced by the standard projection $B_n\to S_n$. (We write our wreath products with the acting group on the left, for convenience. Also, we may sometimes write $\beta(0,\dots,0)$ as $\beta$ and $1_{B_n}(m_1,\dots,m_n)$ as $(m_1,\dots,m_n)$ for the sake of notational elegance.) An \emph{expansion} of this triple is another triple of the form
\[
(T_-',\beta's_k^{m_k}(m_1,\dots,m_{k-1},m_k,m_k,m_{k+1},\dots,m_n),T_+') \text{,}
\]
where $T_+'$ is $T_+$ with a caret added to the $k$th leaf for some $1\le k\le n$, $\beta'$ is $\beta$ with its $k$th ribbon bifurcated into two parallel ribbons and $T_-'$ is $T_-$ with a caret added to the $\rho(\beta)(k)$th leaf. Here $s_k$ is the $k$th standard generator of $B_n$, in the standard presentation
\[
B_n = \left\langle s_1,\dots,s_{n-1} \left| \begin{array}{ll}
     &s_i s_{i+1}s_i = s_{i+1}s_i s_{i+1} \text{ for all } i\text{ and}\\
     &s_i s_j=s_j s_i \text{ for all } i,j \text{ with } |i-j|>1
\end{array}\right. \right\rangle \text{.}
\]
Let us adopt the convention that $s_k$ crosses the $k$th ribbon (counting at the bottom) under the $(k+1)$st ribbon, and a positive single twist of a ribbon involves the left side of the ribbon (looking at the bottom) twisting under the right side. These conventions make the definition of expansion look somewhat natural; see Figure~\ref{fig:twist_expand}.

\begin{figure}[htb]
 \centering
 \begin{tikzpicture}[line width=0.8pt]
  \draw (0,0) -- (1,0)
   (1,0) to [out=-90, in=90, looseness=1] (0,-2);
	\draw[white,line width=4pt]
	 (0,0) to [out=-90, in=90, looseness=1] (1,-2);
	\draw (0,0) to [out=-90, in=90, looseness=1] (1,-2)
	 (0,-2) -- (1,-2);
  \node at (2,-1) {$=$};
  \begin{scope}[xshift=4cm,yshift=1.5cm]
     \draw
      (0,0) -- (1,0)
	  (0,0) to [out=-90, in=90, looseness=1] (-1,-1.5)
	  (1,0) to [out=-90, in=90, looseness=1] (2,-1.5)
	  (0.5,-1) to [out=-180, in=90, looseness=1] (0,-1.5)
	  (0.5,-1) to [out=0, in=90, looseness=1] (1,-1.5);
	 \draw
	  (1,-1.5) to [out=-90, in=90, looseness=1] (-1,-3.5)
		(2,-1.5) to [out=-90, in=90, looseness=1] (0,-3.5);
	 \draw[white,line width=4pt]
	  (-1,-1.5) to [out=-90, in=90, looseness=1] (1,-3.5)
	  (0,-1.5) to [out=-90, in=90, looseness=1] (2,-3.5);
	 \draw[white,line width=10pt]
	  (-.5,-1.5) to [out=-90, in=90, looseness=1] (1.5,-3.5);
	 \draw
	  (-1,-1.5) to [out=-90, in=90, looseness=1] (1,-3.5)
	  (0,-1.5) to [out=-90, in=90, looseness=1] (2,-3.5)
	  (0,-3.5) to [out=-90, in=90, looseness=1] (-1,-4.5)
	  (2,-3.5) to [out=-90, in=90, looseness=1] (1,-4.5);
	 \draw[white,line width=4pt]
	  (-1,-3.5) to [out=-90, in=90, looseness=1] (0,-4.5)
		(1,-3.5) to [out=-90, in=90, looseness=1] (2,-4.5);
	 \draw
	  (-1,-3.5) to [out=-90, in=90, looseness=1] (0,-4.5)
		(1,-3.5) to [out=-90, in=90, looseness=1] (2,-4.5);
	 \draw
      (0,-6) -- (1,-6)
	  (0,-6) to [out=90, in=-90, looseness=1] (-1,-4.5)
	  (1,-6) to [out=90, in=-90, looseness=1] (2,-4.5)
	  (0.5,-5) to [out=-180, in=-90, looseness=1] (0,-4.5)
	  (0.5,-5) to [out=0, in=-90, looseness=1] (1,-4.5);
  \end{scope}
 \end{tikzpicture}
 \caption{Expansion in $rV$. Here we see that $[\cdot,1_{B_1}(1),\cdot] = [\wedge,s_1(1,1),\wedge]$, where $\cdot$ is the trivial tree and $\wedge$ is a single caret.}
 \label{fig:twist_expand}
 \end{figure}
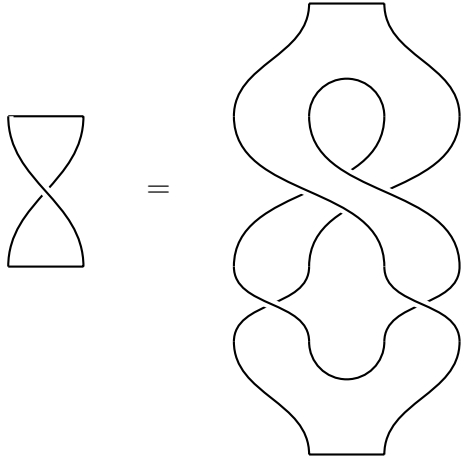

Taking the equivalence relation generated by expansion, we get equivalence classes of the form $[T_-,\beta(m_1,\dots,m_n),T_+]$, which comprise the group $rV$. Just like in $bV$, the group operation is given, roughly, by first expanding until the right tree of the left element equals the left tree of the right element and then canceling these trees. We could consider various subgroups of $rV$ by restricting to pure braids and/or full twists, but for our purposes we will just stick with all braids and all twists. \\

At this point we have
\[
bP < bF < bV < rV \text{,}
\]
where we view $bV$ as the subgroup of $rV$ consisting of elements $[T_-,\beta(0,\dots,0),T_+]$, that is, elements with no twisting. As we have said, $bP$ is normal in $bF$ and $bV$, and in fact it is even normal in $rV$, as we now show:

\begin{lemma}\label{lem:normal}
The subgroup $bP$ is normal in $rV$.
\end{lemma}

\begin{proof}
Let $[U,\gamma,U]\in bP$ and $[T_-,\beta(m_1,\dots,m_n),T_+]\in rV$, expanding so that without loss of generality $U=T_+$. Then we have
\begin{align*}
[T_-,\beta(m_1,\dots,m_n),T_+]&[T_+,\gamma,T_+][T_-,\beta(m_1,\dots,m_n),T_+]^{-1} \\
&= [T_-,\beta(m_1,\dots,m_n),T_+][T_+,\gamma,T_+][T_+,(-m_1,\dots,-m_n)\beta^{-1},T_-] \\
&= [T_-,\beta(m_1,\dots,m_n)\gamma(-m_1,\dots,-m_n)\beta^{-1},T_-] \\
&= [T_-,\beta\gamma\beta^{-1},T_-] \in bP \text{.}
\end{align*}
The last equals sign holds because $\gamma$ is pure and hence $(m_1,\dots,m_n)\gamma=\gamma(m_1,\dots,m_n)$. 
\end{proof}

As we have said, the quotients $bV/bP$ and $bF/bP$ are isomorphic to $V$ and $F$ respectively. The quotient $rV/bP$ is isomorphic to a Thompson-like group constructed analogously to $rV$ but using $S_n \wr \mathbf{Z}$ instead of $B_n \wr \mathbf{Z}$; this could be made more precise by putting a cloning system in the sense of \cite{witzel18} on the family of groups $S_n \wr \mathbf{Z}$, but we will not need to worry about this here. Indeed, all we will need to use $rV/bP$ for later is to relate quasimorphisms of $rV$ to quasimorphisms of $bV$, $bF$, and $bP$, and for this all we need to know about it is the following:

\begin{lemma}\label{lem:rV_mod_bP}
The quotient $rV/bP$ is uniformly perfect.
\end{lemma}

\begin{proof}
Note that $bV/bP \cong V$ is uniformly perfect \cite{uniformlysimple}. Choose $N\in\mathbf{N}$ such that every element of $V$ is a product of at most $N$ commutators. Set $M=(3N+2)$, and we claim that every element of $rV/bP$ is a product of at most $M$ commutators. Let $[T_-,\beta(m_1,\dots,m_n),T_+]\in rV$, and write it as a product of three elements:
\[
[T_-,\beta,T_+][T_+,(m_1,0,\dots,0),T_+][T_+,(0,m_2,\dots,m_n),T_+]\text{.}
\]
Modulo $bP$, we know that this first factor is a product of at most $N$ commutators. The second factor is conjugate to $[T_+,(0,m_1,0,\dots,0),T_+]$, via the conjugator $[T_+,s_1,T_+]$, and this is of the same form as the third factor. Thus it suffices to focus on the third factor, and show that any element of the form $g=[T,(0,m_2,\dots,m_n),T]$ is, modulo $bP$, a product of at most $N+1$ commutators.

Let $T'$ be $T$ with $(n-1)$ new carets added, one after the other, always attaching each new caret to the leftmost leaf. Thus we have $g=[T',(0,\dots,0,m_2,\dots,m_n),T']$, where the number of $0$s is $n$. Let $T''$ be $T$ with $(n-1)$ new carets added, one on each leaf other than the leftmost. Thus we have $g=[T'',\gamma(0,m_2,m_2,m_3,m_3,\dots,m_n,m_n),T'']$, for $\gamma\in B_{2n-1}$ the braid that arises from performing this expansion, namely $\gamma=s_2^{m_2} s_4^{m_3}\cdots s_{2n-2}^{m_n}$. Setting $h=[T'',\gamma,T'']\in bV$ we get $h^{-1}g=[T'',(0,m_2,m_2,m_3,m_3,\dots,m_n,m_n),T'']$. Now let $\alpha\in B_{2n-1}$ be any braid satisfying $\alpha(0,\dots,0,m_2,\dots,m_n)\alpha^{-1}=(0,m_2,0,m_3,\dots,0,m_n,0)$ in $B_{2n-1}\wr\mathbf{Z}$, and set $a=[T'',\alpha,T']$. We get
\begin{align*}
aga^{-1} &= [T'',\alpha,T'][T',(0,\dots,0,m_2,\dots,m_n),T'][T',\alpha^{-1},T'']\\
&= [T'',\alpha(0,\dots,0,m_2,\dots,m_n)\alpha^{-1},T'']\\
&= [T'',(0,m_2,0,m_3,\dots,0,m_n,0),T'']\text{.}
\end{align*}
Hence $h^{-1}gag^{-1}a^{-1} = [T'',(0,0,m_2,0,m_3,\dots,m_{n-1},0,m_n),T'']$. Now using a similar trick as when we conjugated by $a$, this is conjugate to $[T',(0,\dots,0,m_2,\dots,m_n),T']$, which equals $g$. Thus $g$ is conjugate to $h^{-1}gag^{-1}a^{-1}$, and considered modulo $bP$ this is an element of $V$ times a commutator, so we are done.
\end{proof}

It is worth recording the following consequence:

\begin{corollary}\label{cor:rV_perfect}
The group $rV$ is perfect.
\end{corollary}

\begin{proof}
We already know $bV$ is perfect \cite{zaremsky18}, so the derived subgroup $rV'$ contains $bV$. In particular it contains $bP$, and so $rV/rV'$ is a quotient of $rV/bP$. The latter is perfect by Lemma~\ref{lem:rV_mod_bP}, so we conclude $rV=rV'$.
\end{proof}

\subsection{Algebraic properties of $\widehat{bV}$}

Our proof of Theorem~\ref{thm:main2} will rely on some algebraic properties of $bV$ and $\widehat{bV}$, which are the focus of this subsection.

\begin{definition}[Right depth, $\widehat{bV}(1)$]\label{def:right_depth}
Say that the \emph{right depth} of a tree is the distance from its rightmost leaf to its root. Denote by $\widehat{bV}(1) \leq \widehat{bV}$ the subgroup of elements that admit a representative of the form $(T_-, \beta, T_+)$ such that $T_-$ and $T_+$ both have right depth $1$.
\end{definition}

Note that $\widehat{bV}(1)$ is naturally isomorphic to $bV$. Indeed, we have an isomorphism $bV\to \widehat{bV}(1)$ given by $[T_-,\beta,T_+]\to[U_-,\gamma,U_+]$, where $U_-$ is obtained from $T_-$ by adding a new caret whose left leaf is the root of $T_-$, $U_+$ is obtained from $T_+$ by adding a new caret whose left leaf is the root of $T_+$, and $\gamma$ is obtained from $\beta$ by adding one new unbraided strand on the right. This is also discussed in \cite{brady08}.

\begin{definition}[Homomorphism $\chi_1$, subgroup $\widehat{D}$]\label{def:chi1_Khat}
Let $\chi_1\colon \widehat{bV}\to\mathbf{Z}$ be the homomorphism sending $[T_-,\beta,T_+]$ to the right depth of $T_-$ minus the right depth of $T_+$. Since expansions preserve this measurement, thanks to the rightmost strand of such a $\beta$ not braiding, this is well defined, and is clearly a homomorphism. Denote by $\widehat{D}$ the kernel in $\widehat{bV}$ of $\chi_1$.
\end{definition}

We call this map $\chi_1$ since its restriction to Thompson's group $F\le \widehat{bV}$ coincides with a map usually denoted $\chi_1$. Note that $\widehat{D}$ consists of all $[T_-,\beta,T_+]\in\widehat{bV}$ such that $T_-$ and $T_+$ have the same right depth. In particular $\widehat{D}$ contains $\widehat{bV}(1)$. We will see in Corollary~\ref{cor:ablnz_to_1} that $\widehat{D}$ equals the derived subgroup of $\widehat{bV}$.

Recall the usual first generator $x_0$ of Thompson's group $F$. This is the element $x_0=[T_2,1,T_1]$, where $T_i$ is the tree consisting of a caret with a caret attached to its $i$th leaf, and $1$ is the identity in $B_3$. Note that $\chi_1(x_0)=1$. Also note that $x_0^{-1}=[T_1,1,T_2]$.

\begin{lemma}\label{lem:conjugate}
We have $x_0^{-1}\cdot \widehat{bV}(1)\cdot x_0 \le \widehat{bV}(1)$.
\end{lemma}

\begin{proof}
This is clear using strand diagrams; see Figure~\ref{fig:conjugate}. In the figure, we represent an element of $\widehat{bV}(1)$ by drawing the first carets of each tree and the last (unbraided) strand of the braid, and then drawing a gray box to represent the arbitrary remainder of the picture. Now conjugating by $x_0$ and applying some of the equivalence moves from Figure~\ref{fig:strand_moves}, we see that in the resulting strand diagram the trees again have right depth $1$ and the rightmost strand is unbraided (in fact the two rightmost strands are both unbraided).
\end{proof}

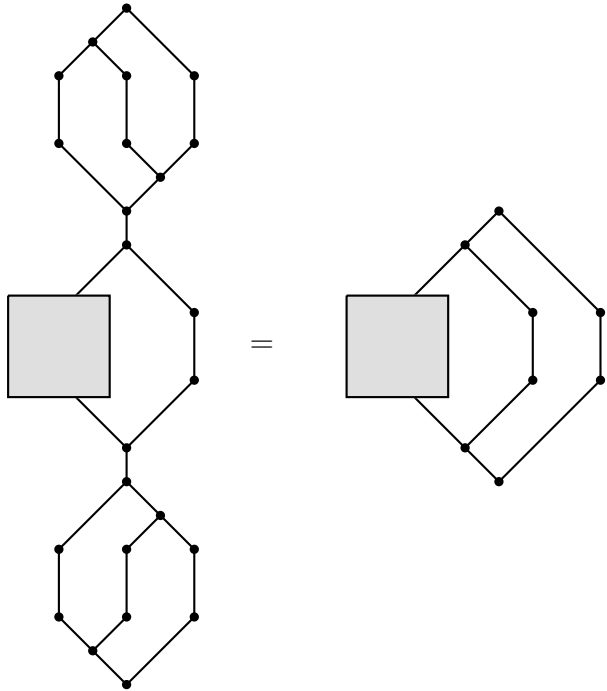
\begin{figure}[htb]
\centering
\begin{tikzpicture}[line width=0.8pt,scale=0.9]

\draw (0,0) -- (1,1) -- (2,0) -- (2,-1) -- (1,-2) -- (0,-1);
\filldraw[lightgray!50] (-0.75,0.25) -- (0.75,0.25) -- (0.75,-1.25) -- (-0.75,-1.25) -- (-0.75,0.25);
\draw (-0.75,0.25) -- (0.75,0.25) -- (0.75,-1.25) -- (-0.75,-1.25) -- (-0.75,0.25)   (1,-2) -- (1,-2.5)   (1,1) -- (1,1.5);
\filldraw (1,1) circle (1.5pt)   (2,0) circle (1.5pt)   (2,-1) circle (1.5pt)   (1,-2) circle (1.5pt);
\node at (3,-0.5) {$=$};

\begin{scope}[yshift=-3.5cm]
\draw (0,0) -- (1,1) -- (2,0) -- (2,-1) -- (1,-2) -- (0,-1) -- (0,0)   (1.5,0.5) -- (1,0) -- (1,-1) -- (0.5,-1.5);
\filldraw (0,0) circle (1.5pt)   (1,1) circle (1.5pt)   (2,0) circle (1.5pt)   (2,-1) circle (1.5pt)   (1,-2) circle (1.5pt)   (0,-1) circle (1.5pt)   (1.5,0.5) circle (1.5pt)   (1,0) circle (1.5pt)   (1,-1) circle (1.5pt)   (0.5,-1.5) circle (1.5pt);
\end{scope}

\begin{scope}[yshift=3.5cm,xshift=2cm,xscale=-1]
\draw (0,0) -- (1,1) -- (2,0) -- (2,-1) -- (1,-2) -- (0,-1) -- (0,0)   (1.5,0.5) -- (1,0) -- (1,-1) -- (0.5,-1.5);
\filldraw (0,0) circle (1.5pt)   (1,1) circle (1.5pt)   (2,0) circle (1.5pt)   (2,-1) circle (1.5pt)   (1,-2) circle (1.5pt)   (0,-1) circle (1.5pt)   (1.5,0.5) circle (1.5pt)   (1,0) circle (1.5pt)   (1,-1) circle (1.5pt)   (0.5,-1.5) circle (1.5pt);
\end{scope}

\begin{scope}[xshift=5cm]
\draw (0,0) -- (1,1) -- (2,0) -- (2,-1) -- (1,-2) -- (0,-1)   (1,1) -- (1.5,1.5) -- (3,0) -- (3,-1) -- (1.5,-2.5) -- (1,-2);
\filldraw[lightgray!50] (-0.75,0.25) -- (0.75,0.25) -- (0.75,-1.25) -- (-0.75,-1.25) -- (-0.75,0.25);
\draw (-0.75,0.25) -- (0.75,0.25) -- (0.75,-1.25) -- (-0.75,-1.25) -- (-0.75,0.25);
\filldraw (1,1) circle (1.5pt)   (2,0) circle (1.5pt)   (2,-1) circle (1.5pt)   (1,-2) circle (1.5pt)   (1.5,1.5) circle (1.5pt)   (3,0) circle (1.5pt)   (3,-1) circle (1.5pt)   (1.5,-2.5) circle (1.5pt);
\end{scope}

\end{tikzpicture}
\caption{The proof of Lemma~\ref{lem:conjugate}: conjugating an element of $\widehat{bV}(1)$ by $x_0$ yields another element of $\widehat{bV}(1)$.}
\label{fig:conjugate}
\end{figure}

\begin{lemma}\label{lem:push_into_1}
For any finite subset $A$ of $\widehat{D}$, there exists $k\ge 0$ such that $x_0^{-k}\cdot A\cdot x_0^k \le \widehat{bV}(1)$.
\end{lemma}

\begin{proof}
Since $A$ is finite, we can choose $k\ge 0$ such that every element of $A$ can be represented by a triple $(T_-,\beta,T_+)$ in which the right depth of $T_-$ (and thus $T_+$) is at most $k+1$. For any such $(T_-,\beta,T_+)$, it is clear that $x_0^{-k}\cdot [T_-,\beta,T_+]\cdot x_0^k \in \widehat{bV}(1)$. See Figure~\ref{fig:push_into_1} for an example.
\end{proof}

\begin{figure}[htb]
\centering
\begin{tikzpicture}[line width=0.8pt,scale=0.8]

\draw (0,0) -- (1.5,1.5) -- (3,0) -- (3,-1.5) -- (1.5,-3) -- (0,-1.5);
\draw (2,1) -- (1,0)   (2.5,0.5) -- (2,0)   (2,-2.5) -- (1,-1.5)   (2.5,-2) -- (2,-1.5);
\filldraw[lightgray!50] (-0.75,0.25) -- (2.5,0.25) -- (2.5,-1.75) -- (-0.75,-1.75) -- (-0.75,0.25);
\draw (-0.75,0.25) -- (2.5,0.25) -- (2.5,-1.75) -- (-0.75,-1.75) -- (-0.75,0.25)   (1.5,-3) -- (1.5,-3.5)   (1.5,1.5) -- (1.5,2);
\filldraw (1.5,1.5) circle (1.5pt)   (2,1) circle (1.5pt)   (2.5,0.5) circle (1.5pt)   (3,0) circle (1.5pt)   (3,-1.5) circle (1.5pt)   (2.5,-2) circle (1.5pt)   (2,-2.5) circle (1.5pt)   (1.5,-3) circle (1.5pt);

\node at (3.75,-0.75) {$=$};

\begin{scope}[yshift=5cm]
\draw (0,0) -- (1.5,1.5) -- (3,0) -- (3,-1.5) -- (1.5,-3) -- (0,-1.5) -- (0,0);
\draw (1,1) -- (2,0)   (0.5,0.5) -- (1,0)   (2,-2.5) -- (1,-1.5)   (2.5,-2) -- (2,-1.5)   (1,0) -- (1,-1.5)   (2,0) -- (2,-1.5);
\filldraw (1.5,1.5) circle (1.5pt)   (1,1) circle (1.5pt)   (0.5,0.5) circle (1.5pt)   (3,0) circle (1.5pt)   (3,-1.5) circle (1.5pt)   (2.5,-2) circle (1.5pt)   (2,-2.5) circle (1.5pt)   (1.5,-3) circle (1.5pt)   (0,0) circle (1.5pt)   (1,0) circle (1.5pt)   (2,0) circle (1.5pt)   (0,-1.5) circle (1.5pt)   (1,-1.5) circle (1.5pt)   (2,-1.5) circle (1.5pt);
\end{scope}

\begin{scope}[yshift=-5cm,xscale=-1,xshift=-3cm]
\draw (0,0) -- (1.5,1.5) -- (3,0) -- (3,-1.5) -- (1.5,-3) -- (0,-1.5) -- (0,0);
\draw (1,1) -- (2,0)   (0.5,0.5) -- (1,0)   (2,-2.5) -- (1,-1.5)   (2.5,-2) -- (2,-1.5)   (1,0) -- (1,-1.5)   (2,0) -- (2,-1.5);
\filldraw (1.5,1.5) circle (1.5pt)   (1,1) circle (1.5pt)   (0.5,0.5) circle (1.5pt)   (3,0) circle (1.5pt)   (3,-1.5) circle (1.5pt)   (2.5,-2) circle (1.5pt)   (2,-2.5) circle (1.5pt)   (1.5,-3) circle (1.5pt)   (0,0) circle (1.5pt)   (1,0) circle (1.5pt)   (2,0) circle (1.5pt)   (0,-1.5) circle (1.5pt)   (1,-1.5) circle (1.5pt)   (2,-1.5) circle (1.5pt);
\end{scope}

\begin{scope}[xshift=5.5cm]
\draw (0,0) -- (1.5,1.5) -- (3,0) -- (3,-1.5) -- (1.5,-3) -- (0,-1.5);
\draw (1,1) -- (2,0)   (0.5,0.5) -- (1,0)   (1,-2.5) -- (2,-1.5)   (0.5,-2) -- (1,-1.5);
\filldraw[lightgray!50] (-0.75,0.25) -- (2.5,0.25) -- (2.5,-1.75) -- (-0.75,-1.75) -- (-0.75,0.25);
\draw (-0.75,0.25) -- (2.5,0.25) -- (2.5,-1.75) -- (-0.75,-1.75) -- (-0.75,0.25);
\filldraw (1.5,1.5) circle (1.5pt)   (1,1) circle (1.5pt)   (0.5,0.5) circle (1.5pt)   (3,0) circle (1.5pt)   (3,-1.5) circle (1.5pt)   (0.5,-2) circle (1.5pt)   (1,-2.5) circle (1.5pt)   (1.5,-3) circle (1.5pt);
\end{scope}

\end{tikzpicture}
\caption{An example of the proof of Lemma~\ref{lem:push_into_1} (for $A$ having just one element): conjugating an element of $\widehat{D}$ in which the trees have right depth $3$ by $x_0^2$ yields an element of $\widehat{bV}(1)$.}
\label{fig:push_into_1}
\end{figure}
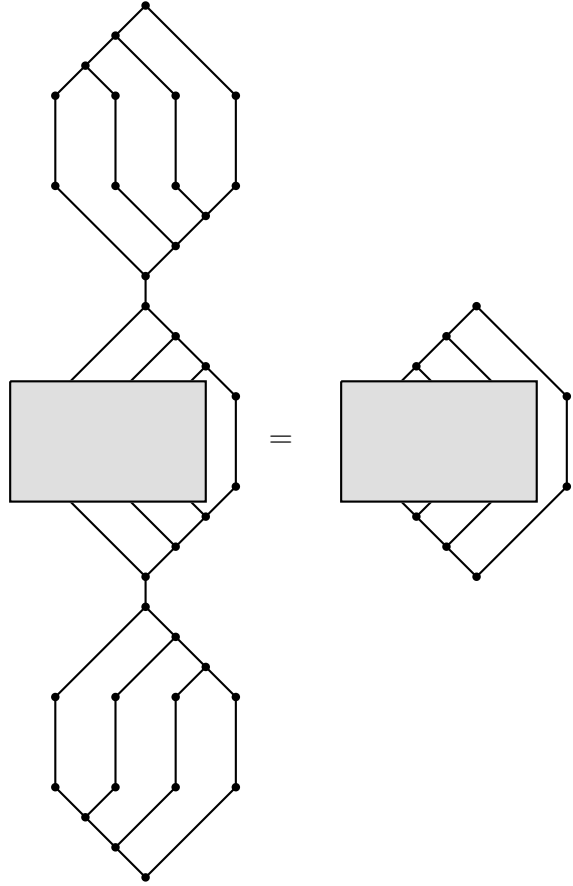

\begin{corollary}\label{cor:hnn}
The group $\widehat{bV}$ is isomorphic to an ascending HNN-extension of $bV$.
\end{corollary}

\begin{proof}
To get our result, we will verify the conditions in \cite[Lemma~3.1]{GMSW}, using $\widehat{bV}(1)$ (which is isomorphic to $bV$) as the base and $x_0$ as the stable letter. Clearly no non-trivial power of $x_0$ lies in $\widehat{bV}(1)$. Lemma~\ref{lem:conjugate} shows that $x_0^{-1}\cdot \widehat{bV}(1)\cdot x_0 \le \widehat{bV}(1)$. Finally, we need to show that $\widehat{bV}$ is generated by $\widehat{bV}(1)$ and $x_0$. Given $[T_-,\beta,T_+]\in \widehat{bV}$, up to right multiplication by a power of $x_0$ we can assume that $[T_-,\beta,T_+]\in\widehat{D}$, i.e., $T_-$ and $T_+$ have the same right depth. Now Lemma~\ref{lem:push_into_1} says we can conjugate by some power of $x_0$ so that our element lands in $\widehat{bV}(1)$.
\end{proof}

\begin{corollary}\label{cor:ablnz_to_1}
The derived subgroup $\widehat{bV}'$ equals $\widehat{D}$, so the abelianisation of $\widehat{bV}$ is $\mathbf{Z}$, given by the map $\chi_1$.
\end{corollary}

\begin{proof}
Since $\widehat{D}$ is the kernel of a map to $\mathbf{Z}$, it contains $\widehat{bV}'$. Conversely, since $\widehat{bV}(1)$ is isomorphic to $bV$, and $bV$ is perfect \cite{zaremsky18}, Lemma~\ref{lem:push_into_1} implies that any element of $\widehat{D}$ is conjugate in $\widehat{bV}$ to an element of a perfect subgroup of $\widehat{bV}$, which shows that every element of $\widehat{D}$ lies in $\widehat{bV}'$. This shows $\widehat{bV}'=\widehat{D}$, and the second statement follows since $\widehat{D}$ is the kernel of $\chi_1$ in $\widehat{bV}$.
\end{proof}

Brin showed in \cite{brin06} that $bV$ and $\widehat{bV}$ are finitely presented. In fact, $bV$ is even of type $\F_\infty$ \cite{bux}. The techniques in \cite{bux} could likely be used to show that $\widehat{bV}$ is also of type $\F_\infty$, but now thanks to Corollary~\ref{cor:hnn} we can prove this much more quickly:

\begin{corollary}\label{cor:F_infty}
The group $\widehat{bV}$ is of type $\F_\infty$.
\end{corollary}

\begin{proof}
It is a standard fact that an ascending HNN-extension of a group of type $\F_n$ is itself of type $\F_n$ (see e.g., \cite[End of Section 2]{BDH}). Since $bV$ is of type $\F_\infty$ \cite{bux}, Corollary~\ref{cor:hnn} implies that $\widehat{bV}$ is as well.
\end{proof}

The key dynamical feature that will make bounded cohomology vanish is contained in the following lemma.

\begin{lemma}\label{lem:commconjbV}
There exists $g \in \widehat{D}$ such that every element of $\widehat{bV}(1)$ commutes with every element of $g^{-1}\cdot \widehat{bV}(1)\cdot g$.
\end{lemma}

\begin{proof}
We define $g = [T_2,s_1,T_2]$, where as before $T_2$ is a caret with a second caret hanging on the right, and $s_1$ is the first standard generator of $B_3$, i.e., the element braiding the first two strands with a single twist. Since $\beta$ does not braid the rightmost strand we have $g \in\widehat{bV}$, and since clearly $\chi_1([T_2,\beta,T_2])=0$ we have $g\in\widehat{D}$. We see in Figure~\ref{fig:comm_conj} that in any element of $g^{-1} \cdot \widehat{bV}(1) \cdot g$, the trees both have ``left depth'' $1$ and the first strand does not braid with anything. Since elements of $\widehat{bV}(1)$ and $g^{-1} \cdot \widehat{bV}(1) \cdot g$ therefore braid disjoint sets of strands, they commute.
\end{proof}

\begin{figure}[htb]
\centering
\begin{tikzpicture}[line width=0.8pt]

\draw (0,0) -- (1,1) -- (2,0) -- (2,-1) -- (1,-2) -- (0,-1);
\filldraw[lightgray!50] (-0.75,0.25) -- (0.75,0.25) -- (0.75,-1.25) -- (-0.75,-1.25) -- (-0.75,0.25);
\draw (-0.75,0.25) -- (0.75,0.25) -- (0.75,-1.25) -- (-0.75,-1.25) -- (-0.75,0.25)   (1,-2) -- (1,-2.5)   (1,1) -- (1,1.5);
\filldraw (1,1) circle (1.5pt)   (2,0) circle (1.5pt)   (2,-1) circle (1.5pt)   (1,-2) circle (1.5pt);
\node at (3,-0.5) {$=$};

\begin{scope}[yshift=-3.5cm]
\draw (0,0) -- (1,1) -- (2,0) -- (2,-1) -- (1,-2) -- (0,-1)   (1.5,0.5) -- (1,0)   (1,-1) -- (1.5,-1.5);
\draw (1,0) to[out=-90,in=90] (0,-1);
\draw[white,line width=4pt] (0,0) to[out=-90,in=90] (1,-1);
\draw (0,0) to[out=-90,in=90] (1,-1);
\filldraw (0,0) circle (1.5pt)   (1,1) circle (1.5pt)   (2,0) circle (1.5pt)   (2,-1) circle (1.5pt)   (1,-2) circle (1.5pt)   (0,-1) circle (1.5pt)   (1.5,0.5) circle (1.5pt)   (1,0) circle (1.5pt)   (1,-1) circle (1.5pt)   (1.5,-1.5) circle (1.5pt);
\end{scope}

\begin{scope}[yshift=3.5cm]
\draw (0,0) -- (1,1) -- (2,0) -- (2,-1) -- (1,-2) -- (0,-1)   (1.5,0.5) -- (1,0)   (1,-1) -- (1.5,-1.5);
\draw (0,0) to[out=-90,in=90] (1,-1);
\draw[white,line width=4pt] (1,0) to[out=-90,in=90] (0,-1);
\draw (1,0) to[out=-90,in=90] (0,-1);
\filldraw (0,0) circle (1.5pt)   (1,1) circle (1.5pt)   (2,0) circle (1.5pt)   (2,-1) circle (1.5pt)   (1,-2) circle (1.5pt)   (0,-1) circle (1.5pt)   (1.5,0.5) circle (1.5pt)   (1,0) circle (1.5pt)   (1,-1) circle (1.5pt)   (1.5,-1.5) circle (1.5pt);
\end{scope}

\begin{scope}[xshift=7cm,xscale=-1]
\draw (0,-1) -- (0,0) -- (1,1) -- (2,0) -- (2,-1) -- (1,-2) -- (0,-1)   (1,1) -- (1.5,1.5) -- (3,0) -- (3,-1) -- (1.5,-2.5) -- (1,-2);
\begin{scope}[xshift=1.5cm]\filldraw[lightgray!50] (-0.75,0.25) -- (0.75,0.25) -- (0.75,-1.25) -- (-0.75,-1.25) -- (-0.75,0.25);
\draw (-0.75,0.25) -- (0.75,0.25) -- (0.75,-1.25) -- (-0.75,-1.25) -- (-0.75,0.25);\end{scope}
\filldraw (1,1) circle (1.5pt)   (0,0) circle (1.5pt)   (0,-1) circle (1.5pt)   (1,-2) circle (1.5pt)   (1.5,1.5) circle (1.5pt)   (3,0) circle (1.5pt)   (3,-1) circle (1.5pt)   (1.5,-2.5) circle (1.5pt);
\end{scope}

\end{tikzpicture}
\caption{An arbitrary conjugate of an element of $\widehat{bV}(1)$ by $g$. We see that it will commute with any element of $\widehat{bV}(1)$.}
\label{fig:comm_conj}
\end{figure}
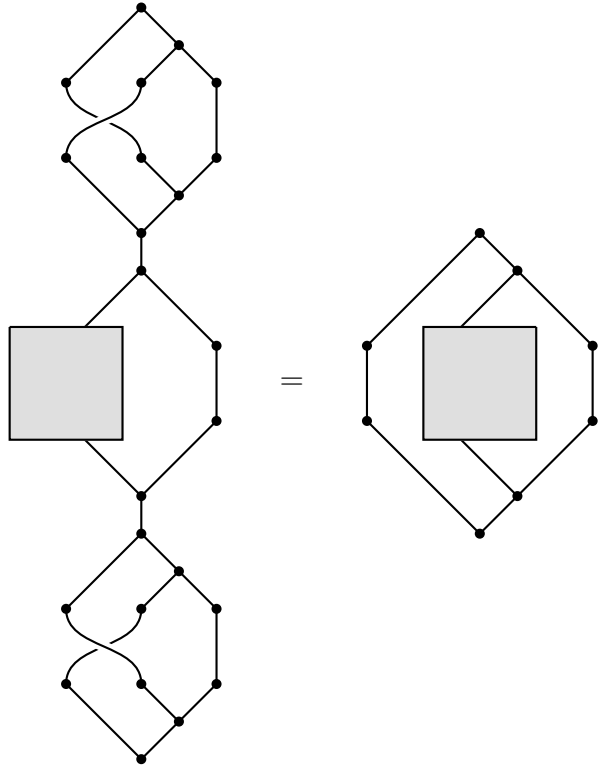

\section{Second bounded cohomology}

We will work only with bounded cohomology with trivial real coefficients, and use the definition in terms of the bar resolution. We refer the reader to \cite{brown} and \cite{frigerio} for a general and complete treatment of ordinary and bounded cohomology of discrete groups, respectively. For the more general setting of locally compact groups, we refer the reader to \cite{monod}. \\

For every $n \geq 0$, denote by $\CC^n(\Gamma)$ the set of real-valued functions on $\Gamma^n$. By convention, $\Gamma^0$ is a single point, so $\CC^0(\Gamma) \cong \mathbf{R}$ consists only of constant functions. We define differential operators $\delta^\bullet : \CC^\bullet(\Gamma) \to \CC^{\bullet+1}(\Gamma)$ as follows:
$\delta^0 = 0$
and for $n \geq 1$:
\begin{equation*}
\begin{split}
    \delta^n(f)(g_1, \ldots, g_{n+1}) & = f (g_2, \ldots, g_{n+1}) \\
    & + \sum\limits_{i = 1}^n (-1)^i f(g_1, \ldots, g_i g_{i+1}, \ldots, g_{n+1}) \\
    & + (-1)^{n+1} f(g_1, \ldots, g_n).
\end{split}
\end{equation*}

One can check that $\delta^{\bullet+1} \delta^{\bullet} = 0$, so $(\CC^\bullet(\Gamma), \delta^\bullet)$ is a cochain complex. We denote by $\ZZ^\bullet(\Gamma) \coloneqq \ker(\delta^\bullet)$ the set of \emph{cocycles}, and by $\BB^\bullet(\Gamma) \coloneqq \im(\delta^{\bullet-1})$ the set of \emph{coboundaries}. The quotient $\HH^\bullet(\Gamma) \coloneqq \ZZ^\bullet(\Gamma) / \BB^\bullet(\Gamma)$ is the \emph{cohomology of $\Gamma$ with trivial real coefficients}. We will also call this the \emph{ordinary cohomology} to make a clear distinction from the bounded one, which we proceed to define. \\

Restricting to functions $f :\Gamma^\bullet \to \mathbf{R}$ that are bounded, meaning that their supremum $\| f \|_\infty$ is finite, leads to a subcomplex $(\CC^\bullet_b(\Gamma), \delta^\bullet)$. We denote by $\ZZ^\bullet_b(\Gamma)$ the \emph{bounded cocycles}, and by $\BB^\bullet_b(\Gamma)$ the \emph{bounded coboundaries}. The vector space $\HH^\bullet_b(\Gamma) \coloneqq \ZZ^\bullet_b(\Gamma) / \BB^\bullet_b(\Gamma)$ is the \emph{bounded cohomology of $\Gamma$} with trivial real coefficients.

The inclusion of the bounded cochain complex into the ordinary one induces a linear map at the level of cohomology, called the \emph{comparison map}:
$$c^\bullet : \HH^\bullet_b(\Gamma) \to \HH^\bullet(\Gamma).$$
This map is in general neither injective nor surjective. In degree $2$, the kernel admits a description in terms of quasimorphisms:

\begin{proposition}[{\cite[Theorem 2.50]{calegari}}]
\label{prop:hbqm}
Let $Q(\Gamma)$ denote the space of quasimorphisms on $\Gamma$ up to bounded distance, and $\ZZ^1(\Gamma)$ the space of homomorphisms $\Gamma \to \mathbf{R}$. Then the sequence
$$0 \to \ZZ^1(\Gamma) \to Q(\Gamma) \xrightarrow{[\delta^1 (-)]} \HH^2_b(\Gamma) \xrightarrow{c^2} \HH^2(\Gamma)$$
is exact. In particular, $c^2$ is injective if and only if every quasimorphism on $\Gamma$ is at a bounded distance from a homomorphism.
\end{proposition}

While many applications of bounded cohomology in geometric group theory, e.g., the study of stable commutator length, are only concerned with quasimorphisms, in different settings the full knowledge of $\HH^2_b$ is of interest. Notable instances include the classification of circle actions \cite{ghys}, and the construction of manifolds with prescribed simplicial volume \cite{spectrum, spectrum2}. \\

In order to prove Theorem~\ref{thm:main2}, that $\widehat{bV}$ has vanishing second bounded cohomology, it is enough to prove this for the subgroup $\widehat{D}$ (from Definition~\ref{def:chi1_Khat}), thanks to the following fact:

\begin{proposition}[{Monod \cite[8.6]{monod}; see also \cite{coamenable}}]\label{prop:amenableext}
Let $n \ge 0$. Let $\Gamma$ be a group and $N$ a normal subgroup such that $\Gamma/N$ is amenable. Then the inclusion $N \to \Gamma$ induces an injection in bounded cohomology $\HH^n_b(\Gamma) \to \HH^n_b(N)$. In particular, if $\HH^n_b(N)=0$ then $\HH^n_b(\Gamma)=0$.
\end{proposition}

To prove vanishing of $\HH^2_b(\widehat{D})$, we will use the following notion.

\begin{definition}
Let $\Gamma$ be a group. We say that $\Gamma$ has \emph{commuting conjugates} if for every finitely generated subgroup $H \leq \Gamma$ there exists $g \in \Gamma$ such that every element of $H$ commutes with every element of $g^{-1}Hg$.
\end{definition}

\begin{theorem}[\cite{FFL1}]\label{thm:commconj}
If $\Gamma$ is a group with commuting conjugates, then $\HH^2_b(\Gamma) = 0$.
\end{theorem}

We are now ready to prove Theorem \ref{thm:main2}.

\begin{proof}[Proof of Theorem \ref{thm:main2}]
Since $\widehat{bV}/\widehat{D} \cong \mathbf{Z}$ by definition, using Proposition \ref{prop:amenableext} it suffices to show that $\HH^2_b(\widehat{D}) = 0$. By Theorem \ref{thm:commconj} it suffices to show that $\widehat{D}$ has commuting conjugates. Let $H \leq \widehat{D}$ be a finitely generated subgroup. By Lemma~\ref{lem:push_into_1}, there exists $k \ge 0$ such that $x_0^{-k}\cdot H\cdot x_0^k \leq \widehat{bV}(1)$. Then by Lemma~\ref{lem:commconjbV}, there exists $g \in \widehat{D}$ such that every element of $g^{-1}\cdot x_0^{-k}\cdot H\cdot x_0^k\cdot g$ commutes with every element of $\widehat{bV}(1)$, and so in particular with every element of $x_0^{-k}\cdot H\cdot x_0^k$. Thus, every element of the conjugate of $H$ by $x_0^k \cdot g\cdot x_0^{-k}$ commutes with every element of $H$. Finally, note that $x_0^k \cdot g\cdot x_0^{-k}\in \widehat{D}$ since $g\in\widehat{D}$, $x_0 \in \widehat{bV}$ and $\widehat{D}$ is normal in $\widehat{bV}$. This shows that $\widehat{D}$ has commuting conjugates and concludes the proof.
\end{proof}

\section{Quasimorphisms on $rV$ and $bV$}\label{sec:quasis}

In this section we prove Theorem~\ref{thm:main1}. We will first work with the ribbon braided Thompson group $rV$ and prove that $Q(rV)$ is infinite-dimensional (Proposition~\ref{prop:rV_inf_dim}), and then prove that unbounded quasimorphisms of $rV$ restrict to unbounded quasimorphisms of $bV$, $bF$ and $bP$ (and indeed, any $\Gamma$ satisfying $bP\le \Gamma\le rV$).

First we need to make the connection between $rV$ and $MCG(\mathbf{R}^2\setminus K)$. This was done implicitly in \cite{funar04,aramayona21}, and more explicitly in \cite[Theorem~3.24]{skipperwustab}. In short, $rV$ is isomorphic to a certain subgroup of mapping classes of $S^2\setminus K$, namely those that are ``asymptotically quasi-rigid'' with respect to some ``rigid structure'' involving choices of ``admissible subsurfaces'' and only act on half of $S^2\setminus K$ in some sense; see \cite[Definition~3.7]{skipperwustab} for all the details. We can view this as describing certain mapping classes of $D^2\setminus K$ that do not require the boundary of $D^2$ to be fixed, but rather allow it to be half-twisted. For an example providing the intuition for how to view an element of $bV$ as a mapping class, see Figure~\ref{fig:bV_as_MCG}.

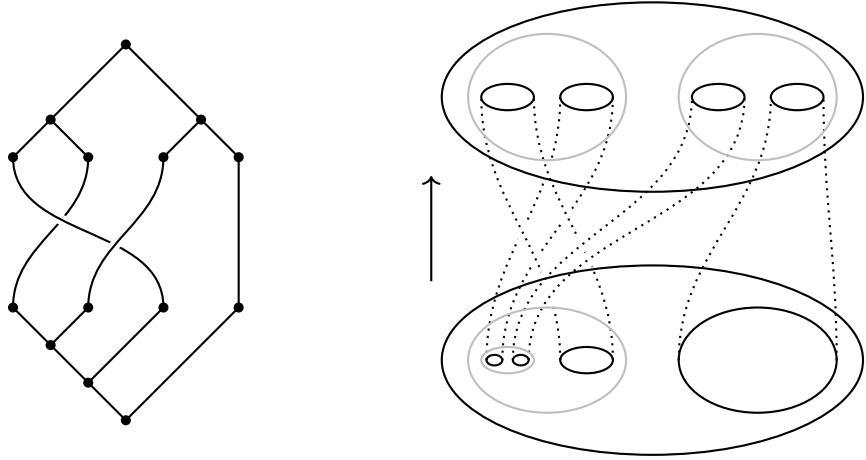
\begin{figure}[htb]
\centering
\begin{tikzpicture}[line width=0.8pt]
\draw (-0.5,-1.5) to [out=-90, in=90] (-1.5,-3.5);
 \draw[line width=4pt, white] (-1.5,-1.5) to [out=-90, in=90] (0.5,-3.5);
 \draw (-1.5,-1.5) to [out=-90, in=90] (0.5,-3.5);
 \draw[line width=4pt, white] (0.5,-1.5) to [out=-90, in=90] (-0.5,-3.5);
 \draw (0.5,-1.5) to [out=-90, in=90] (-0.5,-3.5);
 \draw (1.5,-1.5) -- (1.5,-3.5);
  
 \filldraw (0,0) circle (1.5pt)   (-1.5,-1.5) circle (1.5pt)   (1.5,-1.5) circle (1.5pt)   (-1,-1) circle (1.5pt)   (1,-1) circle (1.5pt)   (-0.5,-1.5) circle (1.5pt)   (0.5,-1.5) circle (1.5pt);
 \draw (-1.5,-1.5) -- (0,0) -- (1.5,-1.5)   (-0.5,-1.5) -- (-1,-1)   (0.5,-1.5) -- (1,-1);
 
 \filldraw (0,-5) circle (1.5pt)   (-1.5,-3.5) circle (1.5pt)   (1.5,-3.5) circle (1.5pt)   (-1,-4) circle (1.5pt)   (-0.5,-4.5) circle (1.5pt)   (-0.5,-3.5) circle (1.5pt)   (0.5,-3.5) circle (1.5pt);;
 \draw (-1.5,-3.5) -- (0,-5) -- (1.5,-3.5)   (-0.5,-3.5) -- (-1,-4)   (0.5,-3.5) -- (-0.5,-4.5);
 
\begin{scope}[xshift=7cm,yshift=-4.2cm,scale=0.7]
\draw[dotted] (-3.15,0) to[out=90,in=-90] (-1.75,5);
\draw[dotted] (-2.85,0) to[out=90,in=-90] (-0.75,5);
\draw[white,line width=0.2cm] (-1.75,0) to[out=90,in=-90] (-3.25,5);
\draw[dotted] (-1.75,0) to[out=90,in=-90] (-3.25,5);
\draw[white,line width=0.2cm] (-0.75,0) to[out=90,in=-90] (-2.25,5);
\draw[dotted] (-0.75,0) to[out=90,in=-90] (-2.25,5);
\draw[white,line width=0.2cm] (-2.65,0) to[out=90,in=-90] (0.75,5);
\draw[dotted] (-2.65,0) to[out=90,in=-90] (0.75,5);
\draw[white,line width=0.2cm] (-2.35,0) to[out=90,in=-90] (1.75,5);
\draw[dotted] (-2.35,0) to[out=90,in=-90] (1.75,5);
\draw[dotted] (0.5,0) to[out=90,in=-90] (2.25,5);
\draw[dotted] (3.5,0) to[out=90,in=-90] (3.25,5);

\draw (0,0) ellipse (4cm and 1.8cm);
\draw[lightgray] (-2,0) ellipse (1.5cm and 1cm);
\draw (2,0) ellipse (1.5cm and 1cm);
\draw[lightgray] (-2.75,0) ellipse (0.5cm and 0.25cm);
\draw (-1.25,0) ellipse (0.5cm and 0.25cm);
\draw (-3,0) ellipse (0.15cm and 0.1cm);
\draw (-2.5,0) ellipse (0.15cm and 0.1cm);

\draw[->] (-4.2,1.5) -- (-4.2,3.5);

\draw (0,5) ellipse (4cm and 1.8cm);
\draw[lightgray] (-2,5) ellipse (1.5cm and 1.2cm);
\draw[lightgray] (2,5) ellipse (1.5cm and 1.2cm);
\draw (-2.75,5) ellipse (0.5cm and 0.25cm);
\draw (-1.25,5) ellipse (0.5cm and 0.25cm);
\draw (2.75,5) ellipse (0.5cm and 0.25cm);
\draw (1.25,5) ellipse (0.5cm and 0.25cm);
\end{scope}

\end{tikzpicture}
\caption{A visualisation of the element $[T_-,\beta,T_+]$ of $bV$ from Figure~\ref{fig:expansion}, as a mapping class on the disk. The bottom (domain) tree $T_+$ describes a decomposition of $D^2$ into the pieces shown, and the top (range) tree $T_-$ describes another such decomposition. The braid $\beta$ then treats the four smallest subdisks in the domain as ``holes,'' with the range viewed similarly, and gives a homeomorphism from the former to the latter, indicated by the dotted lines. This element does not involve twists, but one could picture the holes twisting as well, yielding an element of $rV$.}
\label{fig:bV_as_MCG}
\end{figure}

Now we pass from this picture to $MCG(\mathbf{R}^2\setminus K)$ by viewing $D^2\setminus K$ inside $\mathbf{R}^2\setminus K$, at the expense of modding out the cyclic subgroup generated by a full twist around the boundary of $D^2$. This is represented by the element $[\cdot,1_{B_1}(2),\cdot]$ of $rV$, which generates the center $Z(rV)$, so at this point we have embedded $rV/Z(rV)$ inside $MCG(\mathbf{R}^2\setminus K)$. In particular, we can work in $MCG(\mathbf{R}^2\setminus K)$ to prove that $Q(rV/Z(rV))$ is infinite-dimensional, from which it will immediately follow that $Q(rV)$ is as well. It is also worth mentioning that $bV\cap Z(rV)=\{1\}$, so this provides an explicit embedding of $bV$ into $MCG(\mathbf{R}^2\setminus K)$.

\subsection{Quasimorphisms on $rV$}

The proof that $Q(rV/Z(rV))$ is infinite-dimensional closely follows the proof from \cite{bavard} that $Q(MCG(\mathbf{R}^2 \setminus K))$ is infinite-dimensional. This is Th\'eor\`eme~4.8 of \cite{bavard} (Theorem~4.9 of the English translation \cite{bavard_english}), and we will especially use the constructions from \cite[Subsection~4.1]{bavard}.

In \cite{bavard}, Bavard constructs the so called \emph{ray graph} $X_r$ associated to the surface $\mathbf{R}^2 \setminus K$, and shows that it is hyperbolic. She proceeds to show that the action of $MCG(\mathbf{R}^2 \setminus K)$ on $X_r$ satisfies the hypotheses of Bestvina--Fujiwara's main theorem in \cite{mcgqm}, which implies that $Q(MCG(\mathbf{R}^2 \setminus K))$ is infinite-dimensional. To prove our Proposition~\ref{prop:rV_inf_dim}, we will show that the action of $rV/Z(rV)$ also satisfies these properties, and make reference to \cite[Section 4]{bavard} throughout. \\

We start by reviewing Bavard's proof for $MCG(\mathbf{R}^2 \setminus K)$. By the main theorem of \cite{mcgqm}, it suffices to exhibit elements $h_1, h_2 \in MCG(\mathbf{R}^2 \setminus K)$ with the following properties:
\begin{enumerate}
    \item $h_1$ and $h_2$ are hyperbolic elements for the action of $MCG(\mathbf{R}^2 \setminus K)$ on $X_r$, acting by translation on axes $l_1, l_2$, which are equipped with the orientation of the action of the respective elements.
    \item $h_1$ and $h_2$ are independent, meaning that their fixed point sets in $\partial X_r$ are disjoint.
    \item There exist constants $B, C$ such that for every segment $w$ of $l_2$ longer than $C$, for every $g \in MCG(\mathbf{R}^2 \setminus K)$, if the segment $g \cdot w$ is contained in the $B$-neighbourhood of $l_1$, then it is oriented in the opposite direction.
\end{enumerate}

Identify $K$ with the set $\{0,1\}^{\mathbf{N}}$ of infinite words $\kappa$ in the alphabet $\{0,1\}$, and for each finite word $w\in\{0,1\}^*$ let $K(w)\coloneqq \{w\kappa\mid \kappa\in K\}$ be the \emph{cone} corresponding to $w$. Then define
$$K_0=K(00)\text{, } K_1=K(010)\text{, } K_2=K(0110), \dots, K_\infty=\{0\overline{1}\};$$
$$K_{-1}=K(11)\text{, } K_{-2}=K(101)\text{, } K_{-3}=K(1001), \dots, K_{-\infty}=\{1\overline{0}\}.$$
This provides a partition of $K$ into sets $K_i$ for $-\infty \leq i \leq \infty$, where each $K_i$ for $i \in \mathbf{Z}$ is a clopen set and each $K_{\pm \infty}$ contains one point.

Let us now be more precise about how we would like $K$ to live inside of $\mathbf{R}^2$. Assume that $K$ lies on the horizontal axis $\mathbf{R}$ and is symmetric around $0 \notin K$, and that $K_i \subset \mathbf{R}_{< 0}$ for all $0\le i\le \infty$ and $K_i \subset \mathbf{R}_{> 0}$ for all $-\infty\le i\le -1$. Let $I \subset \mathbf{R}$ be a symmetric open neighbourhood of $0$ that is disjoint from $K$. Finally, let $\mathcal{C}\subseteq \mathbf{R}^2$ be a homeomorphic copy of a circle, formed as the union of a segment in the horizontal axis $\mathbf{R}$ containing all of $K$, and a semicircle in the upper half-plane. Let $\phi$ denote the homeomorphism of $\mathbf{R}^2$ that is a half-turn rotation about the origin, so $\phi$ stabilises $K$, and we will also denote by $\phi$ the mapping class of $\phi$ in $MCG(\mathbf{R}^2\setminus K)$.

\begin{theorem}[{\cite[Th\'eor\`eme~4.8]{bavard}}]
\label{thm:t1}
Let $\widetilde{t}_1$ be any homeomorphism of $\mathbf{R}^2$ that stabilises $\mathcal{C}$, restricts to the identity on $I$ and sends $K_i$ to $K_{i+1}$ for each $i\in\mathbf{Z}$. Let $t_1\in MCG(\mathbf{R}^2 \setminus K)$ be the class of $\widetilde{t}_1$, let $t_2 \coloneqq \phi t_1 \phi^{-1}$, let $h_1 \coloneqq t_1 t_2 t_1$ and let $h_2 \coloneqq \phi h_1^{-1} \phi^{-1}$. Then the elements $h_1$ and $h_2$ satisfy the three properties above, and hence any subgroup of $MCG(\mathbf{R}^2 \setminus K)$ containing $h_1$ and $h_2$ has an infinite-dimensional space of quasimorphisms.
\end{theorem}

As we have seen, $rV$  maps to $MCG(\mathbf{R}^2\setminus K)$ with kernel $Z(rV)\cong \mathbf{Z}$. Note that the image in $MCG(\mathbf{R}^2\setminus K)$ of the element $[\cdot,1_{B_1}(1),\cdot]\in rV$, which is a single half-twist on one ribbon, is precisely the mapping class $\phi$.

We are now ready to prove that $Q(rV)$ is infinite-dimensional.

\begin{proposition}\label{prop:rV_inf_dim}
The space $Q(rV)$ is infinite-dimensional.
\end{proposition}

\begin{proof}
We will prove that $Q(rV/Z(rV))$ is infinite-dimensional, which implies our result. By Theorem~\ref{thm:t1}, it suffices to show that the elements $h_1$ and $h_2$ can be realised inside of $rV/Z(rV)$. Since each of $h_1, h_2$ is obtained as a product of conjugates of $t_1$ by $\phi$ and since $\phi\in rV/Z(rV)$, it suffices to show that $t_1$ can be realised inside $rV/Z(rV)$.

Recall that in Theorem~\ref{thm:t1}, the homeomorphism $\widetilde{t}_1$ representing $t_1$ can be any homeomorphism of $\mathbf{R}^2$ satisfying the following properties:
\begin{enumerate}
    \item $\widetilde{t}_1$ stabilises the topological circle $\mathcal{C}$;
    \item $\widetilde{t}_1|_I$ is the identity;
    \item $\widetilde{t}_1(K_i) = K_{i+1}$ for each $i\in\mathbf{Z}$.
\end{enumerate}
Note that $\widetilde{t}_1$ is defined on all of $\mathbf{R}^2$, so the image $\widetilde{t}_1(K_i)$ is well-defined, even if what we are interested in is a class $t_1 \in MCG(\mathbf{R}^2 \setminus K)$ (this is the usual equivocation between punctures and marked points).

Consider the element $[T_1,s_1^{-1}s_2^{-1}(0,0,-2),T_2]$ of $rV$ represented as in Figure~\ref{fig:MCG}.

\begin{figure}[htb]
    \centering
\begin{tikzpicture}[line width=0.8pt,scale=0.7]
    \draw
      (0,0) -- (1.5,0)
	  (0,0) to [out=-90, in=90, looseness=1] (-1,-1.5)
	  (1.5,0) to [out=-90, in=90, looseness=1] (2.5,-1.5)
	  (0,-1.5) to [out=90, in=90, looseness=1.5] (1.5,-1.5)
	  (2.5,-1.5) -- (2.5,-3)   (1.5,-1.5) -- (1.5,-3);
	\draw
      (-1,-1.5) to [out=-90, in=90, looseness=1] (-2,-3)
	  (0,-1.5) to [out=-90, in=90, looseness=1] (1,-3)
	  (-0.5,-2.5) to [out=-180, in=90, looseness=1] (-1,-3)
	  (-0.5,-2.5) to [out=0, in=90, looseness=1] (0,-3);
	 \draw
	  (-2,-3) to [out=-90, in=90, looseness=1] (1.5,-6)
		(-1,-3) to [out=-90, in=90, looseness=1] (2.5,-6);
	 \draw[white,line width=16pt]
	  (.5,-3) to [out=-90, in=90, looseness=1] (-1.5,-6);
	 \draw
	  (0,-3) to [out=-90, in=90, looseness=1] (-2,-6)
	  (1,-3) to [out=-90, in=90, looseness=1] (-1,-6);
	 \draw[white,line width=16pt]
	  (2,-3) to [out=-90, in=90, looseness=1] (0,-6);
	 \draw
	  (1.5,-3) to [out=-90, in=90, looseness=1] (-.5,-6)
	  (2.5,-3) to [out=-90, in=90, looseness=1] (0.5,-6);
	 \draw
	  (1.5,-6) to [out=-90, in=90, looseness=1] (2.5,-7);
	 \draw[white,line width=4pt]
	  (2.5,-6) to [out=-90, in=90, looseness=1] (1.5,-7);
	 \draw
	  (2.5,-6) to [out=-90, in=90, looseness=1] (1.5,-7);
	 \draw
	  (1.5,-7) to [out=-90, in=90, looseness=1] (2.5,-8);
	 \draw[white,line width=4pt]
	  (2.5,-7) to [out=-90, in=90, looseness=1] (1.5,-8);
	 \draw
	  (2.5,-7) to [out=-90, in=90, looseness=1] (1.5,-8);
	 \draw (-2,-6) -- (-2,-8)   (-1,-6) -- (-1,-8)   (-.5,-6) -- (-.5,-8)   (.5,-6) -- (.5,-8);
	 \begin{scope}[yshift=-11cm,xshift=0.5cm,xscale=-1,yscale=-1]
	 \draw
      (0,0) -- (1.5,0)
	  (0,0) to [out=-90, in=90, looseness=1] (-1,-1.5)
	  (1.5,0) to [out=-90, in=90, looseness=1] (2.5,-1.5)
	  (0,-1.5) to [out=90, in=90, looseness=1.5] (1.5,-1.5)
	  (2.5,-1.5) -- (2.5,-3)   (1.5,-1.5) -- (1.5,-3);
	\draw
      (-1,-1.5) to [out=-90, in=90, looseness=1] (-2,-3)
	  (0,-1.5) to [out=-90, in=90, looseness=1] (1,-3)
	  (-0.5,-2.5) to [out=-180, in=90, looseness=1] (-1,-3)
	  (-0.5,-2.5) to [out=0, in=90, looseness=1] (0,-3);
	 \end{scope}
	  
\begin{scope}[xshift=10cm,yshift=-9.5cm]

\draw[dotted] (3.25,1.75) -- (2.25,2.5);
\draw[dotted] (3.25,2.5) -- (2.25,3.25);
\draw[white,line width=0.27cm] (2.25,1.75) -- (3.25,2.5);
\draw[dotted] (2.25,1.75) -- (3.25,2.5);
\draw[white,line width=0.27cm] (2.25,2.5) -- (3.25,3.25);
\draw[dotted] (2.25,2.5) -- (3.25,3.25);

\draw[dotted] (2.25,3) to[out=90,in=-90] (-3.25,8);
\draw[dotted] (3.25,3) to[out=90,in=-90] (-2.25,8);
\draw[dotted] (2.25,0) -- (2.25,3)   (3.25,0) -- (3.25,3);
\draw[white,line width=0.2cm] (-1.75,5) -- (-1.75,8);
\draw[dotted] (-1.75,5) -- (-1.75,8);
\draw[dotted] (-3.5,0) to[out=90,in=-90] (-1.75,5);
\draw[white,line width=0.2cm] (-.5,0) to[out=90,in=-90] (-.75,8);
\draw[dotted] (-.5,0) to[out=90,in=-90] (-.75,8);
\draw[white,line width=0.2cm] (.75,0) to[out=90,in=-90] (.5,8);
\draw[dotted] (.75,0) to[out=90,in=-90] (.5,8);
\draw[white,line width=0.2cm] (1.75,4) to[out=90,in=-90] (3.5,8);
\draw[dotted] (1.75,4) to[out=90,in=-90] (3.5,8);
\draw[dotted] (1.75,0) -- (1.75,4);

\draw[SkyBlue,line width=0.1cm] (-3.75,0) to[out=90,in=90,looseness=.7] (3.75,0);
\draw[red,line width=0.2cm] (-.3,0) -- (.3,0);
\draw[Peach,line width=0.2cm] (.3,0) -- (1.9,0);
\draw[Goldenrod,line width=0.2cm] (1.9,0) -- (2.1,0);
\draw[LimeGreen,line width=0.2cm] (2.1,0) -- (3.82,0);
\draw[Orchid,line width=0.2cm] (-3.82,0) -- (-0.3,0);

\draw[Goldenrod,line width=0.1cm] (-3.75,8) to[out=90,in=90,looseness=.7] (3.75,8);
\draw[red,line width=0.2cm] (-.3,8) -- (.3,8);
\draw[Peach,line width=0.2cm] (.3,8) -- (3.82,8);
\draw[LimeGreen,line width=0.2cm] (-3.82,8) -- (-2.1,8);
\draw[SkyBlue,line width=0.2cm] (-2.1,8) -- (-1.9,8);
\draw[Orchid,line width=0.2cm] (-1.9,8) -- (-.3,8);

\draw[dotted] (2.25,0) -- (2.25,3)   (3.25,0) -- (3.25,3);
\draw[dotted] (-1.75,5) -- (-1.75,8);
\draw[dotted] (-3.5,0) to[out=90,in=-90] (-1.75,5);
\draw[dotted] (-.5,0) to[out=90,in=-90] (-.75,8);
\draw[dotted] (.75,0) to[out=90,in=-90] (.5,8);
\draw[dotted] (1.75,4) to[out=90,in=-90] (3.5,8);
\draw[dotted] (1.75,0) -- (1.75,4);

\draw (0,0) ellipse (4cm and 1.8cm);
\draw (-2,0) ellipse (1.5cm and 1cm);
\draw[lightgray] (2,0) ellipse (1.5cm and 1cm);
\draw (2.75,0) ellipse (0.5cm and 0.25cm);
\draw (1.25,0) ellipse (0.5cm and 0.25cm);

\draw[->] (-4.2,2) -- (-4.2,6);

\draw (0,8) ellipse (4cm and 1.8cm);
\draw[lightgray] (-2,8) ellipse (1.5cm and 1cm);
\draw (2,8) ellipse (1.5cm and 1cm);
\draw (-2.75,8) ellipse (0.5cm and 0.25cm);
\draw (-1.25,8) ellipse (0.5cm and 0.25cm);
\end{scope}

\end{tikzpicture}
\caption{The desired element. On the left is a ribbon strand diagram representing this element. On the right is the corresponding mapping class, as in Figure~\ref{fig:bV_as_MCG}. We indicate the circle $\mathcal{C}$ with rainbow colours, to make it clear how the mapping class acts on it. The fixed interval $I$ is red.}
\label{fig:MCG}
\end{figure}
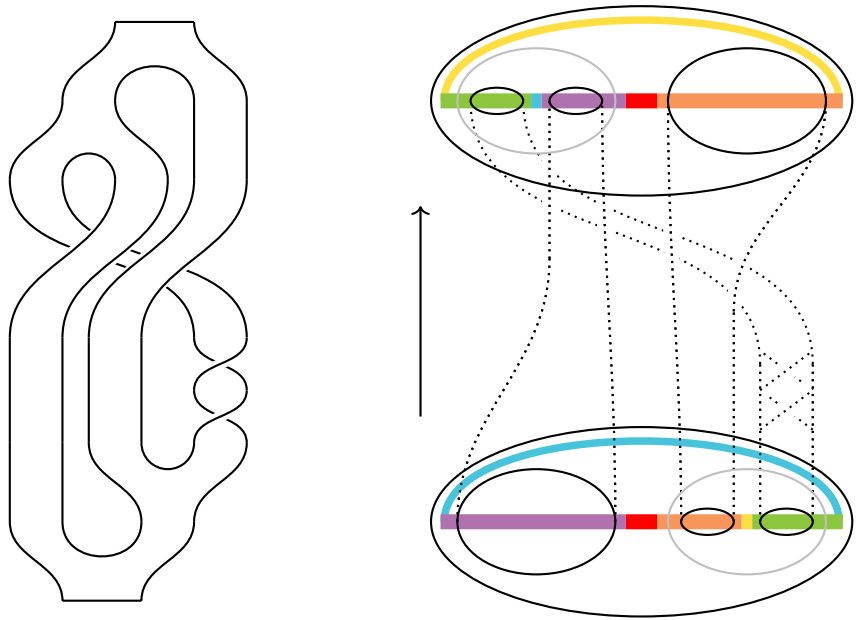

With $\mathcal{C}$ and $I$ indicated in the picture, it is clear that up to isotopy $\mathcal{C}$ is stabilised (thanks to the third strand twisting), and that $I$ is fixed pointwise. One can also check that $K_i$ is sent to $K_{i+1}$ for each $i\in\mathbf{Z}$. We conclude that all the criteria are satisfied, and so we are done.
\end{proof}

\subsection{Quasimorphisms on $bV$}

The final step in the proof of Theorem \ref{thm:main1} is to show that the quasimorphisms on $rV$ constructed in the previous subsection restrict to non-trivial quasimorphisms on $bV$, $bF$, and $bP$. This will be a consequence of the following general statement applied to $rV$ and $bP$, which follows from left exactness of $Q$; see \cite[Remark~2.90]{calegari}.

\begin{lemma}
\label{lem:restr}
Let $\Gamma$ be a group and $\Lambda \leq \Gamma$ a normal subgroup, and suppose that $Q(\Gamma/\Lambda) = 0$. Then the restriction $Q(\Gamma) \to Q(\Lambda)$ is injective.
\end{lemma}

\begin{proof}[Proof of Theorem \ref{thm:main1}]
Let $\Gamma$ be any group such that $bP \leq \Gamma \leq rV$, for instance any of the groups in the statement of the theorem. Note that the quotient $rV / bP$  is uniformly perfect by Lemma~\ref{lem:rV_mod_bP}, so every quasimorphism on $rV / bP$ is bounded \cite[Lemma 2.2.4]{calegari}, i.e., $Q(rV / bP) = 0$. Hence Lemma~\ref{lem:restr} applies, and the restriction $Q(rV) \to Q(bP)$ is injective. Since this map factors through $Q(rV) \to Q(\Gamma)$, this restriction is also injective. We conclude by Proposition \ref{prop:rV_inf_dim}.
\end{proof}

Since every quasimorphism on a uniformly perfect group is bounded \cite[Lemma 2.2.4]{calegari}, an immediate corollary of Theorem \ref{thm:main1} is the following.

\begin{corollary}
\label{cor:up}
The group $bV$ is not uniformly perfect.\qed
\end{corollary}

Of course we also conclude that $rV$ is not uniformly perfect, despite being perfect (Corollary~\ref{cor:rV_perfect}). Note that $bF$ is not perfect (it has abelianisation $\mathbf{Z}^4$), but it follows from \cite{zaremsky18} that $bF'$ is perfect. However, we can deduce in the same way:

\begin{corollary}
\label{cor:up2}
The group $bF'$ is not uniformly perfect.\qed
\end{corollary}

We should also mention another group fitting between $bP$ and $rV$ and thus having infinite-dimensional space of quasimorphisms, namely the ``braided $T$'' group from \cite{witzel19}. This is the subgroup of $bV$ consisting of elements $[T_-,\beta,T_+]$ such that $\rho(\beta)\in S_n$ is a cyclic permutation. \\

Let us discuss restricting quasimorphisms of $bV$ to $\widehat{bV}$. If $\Gamma \leq MCG(\mathbf{R}^2 \setminus K)$ has a bounded orbit in $X_r$, then the quasimorphisms produced via this action are bounded on $\Gamma$. This is analogous to the behaviour of finite type mapping class groups, in particular for the braid group \cite{feller}, and was already noted by Calegari for $MCG(\mathbf{R}^2 \setminus K)$ \cite{blogpost}. We see that in fact this happens for $\widehat{bV}$, since it fixes the isotopy class of the ray going from the rightmost point of the Cantor set to infinity on the right. Thanks to Theorem \ref{thm:main2}, we can actually prove a stronger version of this statement. Namely, not only are these quasimorphisms bounded on $\widehat{bV}$, but the same is true for every quasimorphism of $bV$.

\begin{corollary}
\label{cor:imageofhat}
For every quasimorphism $q$ of $bV$, the image of $\widehat{bV}$ under $q$ is bounded.
\end{corollary}

\begin{proof}
Let $\psi=[\cdot,1_{B_1}(1),\cdot]$, so viewing $rV/Z(rV)$ as a subgroup of $MCG(\mathbf{R}^2\setminus K)$ we have $\psi Z(rV)=\phi$. Note that the conjugate $\psi^{-1}\widehat{bV}\psi$ equals the subgroup of $bV$ consisting of all elements where the \emph{left}most strand does not braid with anything. Let $\chi_0\colon \psi^{-1}\widehat{bV}\psi \to \mathbf{Z}$ be the map sending $\psi^{-1}g\psi$ to $\chi_1(g)$. Since conjugation by $\psi$ is an isomorphism, Corollary~\ref{cor:ablnz_to_1} implies that the kernel of $\chi_0$ equals the derived subgroup $(\psi^{-1}\widehat{bV}\psi)'$.

Let $g \in \widehat{bV}$, and choose $h \in \widehat{bV}\cap \psi^{-1}\widehat{bV}\psi$ such that $\chi_1(h)=\chi_1(g)$ and $\chi_0(h)=0$. In particular, $h$ lies in $(\psi^{-1}\widehat{bV}\psi)'$ and $gh^{-1}$ lies in $\widehat{bV}'$. By Corollary~\ref{cor:main}, there exists a scalar $\lambda \in \mathbf{R}$ such that $q|_{\widehat{bV}}$ is at a bounded distance $r(q)$ from $\lambda \cdot \chi_1$, where $\chi_1$ is the abelianisation map of $\widehat{bV}$ (Corollary \ref{cor:ablnz_to_1}). It follows that $$|q(gh^{-1})| \leq |\lambda \cdot \chi_1(gh^{-1})| + r(q) = r(q).$$
We also get $|q(h)|\le r(q)$, by the same argument applied to $\psi^{-1}\widehat{bV}\psi$ (which is isomorphic to $\widehat{bV}$), up to taking a larger $r(q)$. Thus
$$|q(g)| \leq |q(gh^{-1})| + |q(h)| + D(q) \leq 2 r(q) + D(q).$$
This shows that $q|_{\widehat{bV}}$ is bounded, which concludes the proof.
\end{proof}

Together with Theorem \ref{thm:main1}, this implies that, analogously to Corollary \ref{cor:up}, there is no uniform-length factorisation for elements of $bV$ in terms of conjugates of elements in $\widehat{bV}$. Indeed, if such a factorisation did exist, we could run a similar argument as in the proof of Corollary~\ref{cor:imageofhat}, and obtain that every quasimorphism of $bV$ is bounded. \\

As one last indication of braided Thompson groups exhibiting unusual bounded cohomological behaviour, consider the short exact sequence $1\to \widehat{bV}\cap bP \to \widehat{bV} \to \widehat{V} \to 1$. By Theorem~\ref{thm:main2}, $\HH_b^2(\widehat{bV})=0$, and in fact the proof works using permutations instead of braids, \emph{mutatis mutandis}, to show that $\HH_b^2(\widehat{V})=0$ (also, it is true in general that a quotient of a group with vanishing second bounded cohomology itself has vanishing second bounded cohomology \cite{bouarich}). However, by Theorem~\ref{thm:main1} we have that $Q(bP)$ is infinite-dimensional, and thus so is $Q(\widehat{bV}\cap bP)$, since $\widehat{bV}\cap bP$ surjects onto $bP$ (Lemma~\ref{lem:onto_bP}); in particular $\HH^2_b(\widehat{bV} \cap bP)$ is infinite-dimensional.

This gives a concrete example of the failure of a $2$-out-of-$3$ property for vanishing of second bounded cohomology: if a group $\Gamma$ has vanishing second bounded cohomology, and a quotient $\Gamma/N$ has the same property, then the kernel $N$ can still have infinite-dimensional second bounded cohomology. For comparison, if $\HH^2_b(N) = 0$, then $\HH^2_b(\Gamma) = 0$ if and only if $\HH^2_b(\Gamma/N) = 0$ \cite[Corollary 4.2.2]{moraschiniraptis}. The failure of this $2$-out-of-$3$ property was observed in \cite[Theorem 4.5]{binate} for every degree, but this is to our knowledge the first ``naturally occurring'' example in degree $2$, as well as the first finitely generated one (and even type $\F_\infty$).


\bibliographystyle{abbrv}
\bibliography{references}

\noindent{\textsc{Department of Mathematics, ETH Z\"urich, Switzerland}}

\noindent{\textit{E-mail address:} \texttt{francesco.fournier@math.ethz.ch}} \\

\noindent{\textsc{Department of Mathematics,
University of \Hawaii at \Manoa.}}

\noindent{\textit{E-mail address:} \texttt{lodha@hawaii.edu}} \\

\noindent{\textsc{Department of Mathematics and Statistics, University at Albany (SUNY)}}

\noindent{\textit{E-mail address:} \texttt{mzaremsky@albany.edu}}

\end{document}